\documentclass{scrartcl}
\usepackage{graphicx} % Required for inserting images
\usepackage[usenames,dvipsnames]{xcolor}
\usepackage{mathtools, amsmath, amssymb, amsthm}
\usepackage{dsfont}
\usepackage[numbers]{natbib}
\usepackage{tikz}

\usepackage{pgfplots}
\pgfplotsset{compat=1.18}
\usetikzlibrary{pgfplots.groupplots}

\usepackage{soul}

\usetikzlibrary{decorations.pathreplacing}
\usepackage{caption}

\newcommand{\defeq}{\vcentcolon=}
\newcommand{\ER}{Erd\H{o}s-Rényi }

\newcommand{\1}{\mathds{1}}
\newcommand{\IR}{\mathds{R}}

\newcommand{\IP}{\mathbb{P}}

\newcommand{\Z}{\mathds{Z}}
\newcommand{\IE}{\mathds{E}}

\newcommand{\N}{\mathbb{N}}

\newcommand{\cA}{\mathcal{A}}

\newcommand{\cC}{\mathcal{C}}

\newcommand{\cE}{\mathcal{E}}

\newcommand{\cN}{\mathcal{N}}

\newcommand{\cP}{\mathcal{P}}

\newcommand{\bfC}{\mathbf{C}}

\newcommand{\bfE}{\mathbf{E}}
\newcommand{\bfG}{\mathbf{G}}

\newcommand{\bfZ}{\mathbf{Z}}
\newcommand{\bfhZ}{\widehat{\bfZ}}
\newcommand{\bfhE}{\widehat{\bfE}}
\newcommand{\bfhG}{\widehat{\bfG}}

\newcommand{\hD}{\widehat{D}}
\newcommand{\hE}{\widehat{E}}
\newcommand{\hS}{\widehat{S}}
\newcommand{\hT}{\widehat{T}}
\newcommand{\hU}{\widehat{U}}
\newcommand{\hV}{\widehat{V}}
\newcommand{\hZ}{\widehat{Z}}

\definecolor{Felix}{rgb}{0.0, 0.5, 0.0}

\newcommand{\gucci}[1]{{\color{blue} $\checkmark$}}

\theoremstyle{plain}

\newtheorem{theorem}{Theorem}[section]       % resets enumeration with new section

\newtheorem{definition}[theorem]{Definition}    % enumeration depends on the counter of theorem
\newtheorem{lemma}[theorem]{Lemma}

\newtheorem{proposition}[theorem]{Proposition}

\newtheorem{remark}[theorem]{Remark}
\newtheorem{conjecture}[theorem]{Conjecture}

\title{The Offended Voter Model}
\author{Raphael Eichhorn\footnote{University of Lübeck, Ratzeburger Allee 160, 23562 Lübeck, Germany, \texttt{raphael.eichhorn@uni-luebeck.de}} , Felix Hermann\footnote{Goethe University Frankfurt, Robert-Mayer-Straße 10, 60486 Frankfurt am Main, Germany, \texttt{hermann@math.uni-frankfurt.de}, \texttt{seiler@math.uni-frankfurt.de}} , $\text{Marco Seiler}^\dagger$}
\date{\today}

\begin{document}
\maketitle
\begin{abstract}
We study a variant of the voter model on a coevolving network in which interactions of two individuals with differing opinions only lead to an agreement on one of these opinions with a fixed probability $q$.
Otherwise, with probability $1-q$, both individuals become offended in the sense that they never interact again, i.e.\ the corresponding edge is removed from the underlying network.
Eventually, these dynamics reach an absorbing state at which there is only one opinion present in each connected component of the network.
If globally both opinions are present at absorption we speak of ``segregation'', otherwise of ``consensus''.
We rigorously show that segregation and a weaker form of consensus both occur with positive probability for every $q \in (0,1)$ and that the segregation probability tends to $1$ as $q \to 0$.
Furthermore, we establish that, if $q \to 1$ fast enough, with high probability the population reaches consensus while the underlying network is still densely connected.
We provide results from simulations to assess the obtained bounds and to discuss further properties of the limiting state.\\
{\footnotesize\textit{2020 Mathematics Subject Classification --} Primary 60K35, Secondary 05C80, 60K37
	}
\end{abstract}
\section{Introduction}

\paragraph*{Motivation} There is a variety of fundamental processes in society which can be modelled via stochastic dynamics evolving on complex networks.
Opinion dynamics in spatially structured populations, spread of misinformation in social media, and the course of an epidemic are only some examples that are highly relevant in applications. This led to a significant research effort in the field of random graphs, which are used to describe complex networks, as well as in the area of stochastic processes evolving on such networks.

We are especially interested in modelling opinion dynamics and how such an exchange of opinions forms a social structure and could even lead to segregation of the network into separate groups which hold different opinions. A simplistic prototype to model opinion dynamics in a spatially structured population is the well-known voter model, which was introduced by Holley and Liggett \cite{holley1975ergodic}. They already established a duality relation between the voter model and a system of coalescing random walks, which is one of the most important tools to study this model. They considered a population of infinite size and showed in particular that coexistence of distinct opinions can only occur if and only if a symmetric random walk on the underlying graph is transient. Note that coexistence was defined as existence of a non-trivial invariant measure, i.e.\ a measure under which with positive probability both opinions appear at the same time.

The initial work of Holley and Liggett inspired others to study the question of coexistence in several variants of the voter model, see for example \cite{liggett1994coexistence,sturm2008heterozy,huo2019zealot}. 
More recently Astoquillca \cite{astoquillca2024stationary} considered voter dynamics on an autonomously evolving graph structure. However, in any case this notion of coexistence can only be observed on infinite graphs since a population described by a connected graph of finite size will reach consensus eventually.

Thus, in order to model segregation it is necessary that opinions as well as the graph structure evolve simultaneously. This motivated us to consider a coevolving variant of the voter model which we call \emph{the offended voter model}.
\paragraph*{Model} We start with a given graph and a configuration of opinions, for example a complete graph where half of the individuals have one opinion and the other half has the counter-opinion. In every discrete time step we choose uniformly at random a connected pair of individuals who disagree. Then, with probability $q\in (0,1)$ these two agree on one of the two opinions, with each outcome being equally likely. On the other hand, if they do not reach agreement, they are both offended and cut ties. As a consequence, we delete the edge between these two individuals permanently and they never interact with each other again. As these dynamics progress, either consensus is reached while the population stays fully connected or the population structure ends up being disconnected such that there are multiple disjoint connected components. The components each reach consensus `locally' but potentially hold opposing opinions.

As far as we know, we are the first to treat this model in a mathematically rigorous manner. However, it has been considered before by Gil and Zanette \cite{gil2006covelution,zanette2006opinion} in a simulation based study. We provide an asymptotic lower bound on the probability that segregation occurs in the offended voter model in the sense that in the final state the induced graph is disconnected and both opinions are still present in the population. In particular, we obtain that with positive probability segregation occurs for all $q\in (0,1)$. In fact, we provide simulations which suggest that this is not only a lower bound, but the exact limit. Furthermore, the probability that in the final state both opinions are still present and of the same order is as well positive for all $q\in (0,1)$. We also show that if $q_N\to 0$ as $N\to \infty$ the induced graph ends up being disconnected with high probability and both opinions are still present with almost the same proportion as initially. On the other hand, as $q_N\to 1$ consensus is reached and the final graph is even densely connected.
\paragraph*{Challenges}
Such coevolving voter models are much harder to treat since the opinion dynamics and graph evolution depend on each other and neither is autonomously evolving. Because of that, some extremely useful properties such as monotonicity are lost. One problematic consequence of this is that the duality relation with coalescing random walks does no longer hold. One of the few works treating such models in a mathematically rigorous manner is by Basu and Sly~\cite{basu2017evolving}. They consider a closely related coevolving voter model, where the initial population structure is given by a dense \ER graph and both opinions are initially present with proportion $\tfrac{1}{2}$. In this model edges between disagreeing individuals are not deleted but rewired. Similarly as in the offended voter model, the final state is either global consensus while all individuals are still connected or segregation. A major difference between these two models is that this rewiring procedure keeps the number of edges present in the graph constant, whereas in our model the number of edges declines as time progresses. This leads to some remarkable differences which we will further discuss in Section~\ref{sec:DiscussionAndSimulation}.

\paragraph*{Structure of the paper}
The rest of the paper is organized as follows: In Section~\ref{sec:model_and_notation} we rigorously define the model and classify the absorbing states.
Section~\ref{sec:results} states our main results comprised of several asymptotic bounds for events on the final state of the offended voter model.
The discussion in Section~\ref{sec:DiscussionAndSimulation} contains a comparison of our findings to the work of Basu and Sly~\cite{basu2017evolving} as well as simulations that validate our results, give rise to several conjectures, show interesting effects and thus motivate future work. Furthermore, we discuss possible model extensions.
Section \ref{sec:construction} prepares the proofs carried out in Section \ref{sec:proofs} by providing two constructions of the model and an auxiliary process.

\section{Model and Notation}\label{sec:model_and_notation}
We denote the complete graph of size $N\in \N$ by $G^N = (\Lambda^N, E^N\big)$ where $\Lambda^N \defeq \{1, \dots, N\}$ is the vertex set and \ $E^N:=\big\{\{x,y\}\subset \Lambda^N\times \Lambda^N: x\neq y\big\}$ is the edge set. Further, we fix $q_N\in[0,1]$. We interpret every vertex $x \in \Lambda^N$ as an individual which has one out of two different opinions, either $0$ or $1$. Roughly speaking, we study the opinion dynamics in a population in which interacting individuals either successfully agree on one of their opinions, with probability $q_N$, or otherwise cease to interact.

We model this by a discrete-time Markov chain $\bfZ^N=\big(\bfZ^N_n\big)_{n \in \N_0}$ taking values in the state space $\cP (\Lambda^N \cup E^N)$ where $\cP$ denotes the power set. We say an individual $x \in \Lambda^N$ has opinion $1$ at time $n$ if $x \in \bfZ^N_n$, and opinion $0$ otherwise. Furthermore, two individuals $x$ and $y$ can interact with each other at time $n$ if $\{x, y\}\in \bfZ^N_n$, i.e.\ the connecting edge is present. 
We call an edge $\{x,y\}$ between individuals $x$ and $y$ \emph{discordant} if they do not have the same opinion and denote the set of discordant edges at time $n$ by
\begin{equation*}
	\bfE_n^{N, d} \defeq \Big\{ \{x,y\} \in \bfZ_n^N : \big|\bfZ_n^N \cap \{x,y\}\big|=1\Big\}.
\end{equation*} 
Let $\bfZ_0^N \subset \Lambda^N \cup E^N$ be the initial state, which is possibly random but independent of the model dynamics. At any time $n\geq 1$, choose a discordant edge $\{x,y\} \in \bfE_{n-1}^{N, d}$ uniformly at random and then independently, with probability
\begin{itemize}
	\item $\frac{q_N}{2}$ both individuals of the edge adopt opinion $1$, i.e. $\bfZ_n^N = \bfZ_{n-1}^N \cup \{x,y\}$,
	\item $\frac{q_N}{2}$ both individuals of the edge adopt opinion $0$, i.e. $\bfZ_n^N = \bfZ_{n-1}^N \setminus \{x,y\}$,
	\item $1-q_N$ the edge is removed, i.e. $\bfZ_n^N = \bfZ_{n-1}^N \setminus \big\{\{x,y\}\big\}$.
\end{itemize}
We refer to events of the first and second type as \emph{opinion update} and to events of the third type as \emph{deletion}. The process $\bfZ^N$ is called the \emph{offended voter model}, abbreviated to OV-Model, and $q_N \in [0,1]$ the \emph{voting probability}. Furthermore, we denote the randomly evolving population graph at time $n$ associated with the model by ${\bfG^N_n = (\Lambda^N, \bfZ_n^N \cap E^N)}$.

Observe that if at some time $n$ there are no longer any discordant edges left, i.e.\ $\bfE_n^{N, d}=\emptyset$, then the process stays constant and has thus entered an absorbing state. %Even through $\bfE_n^{N, d}$ is not monotone decreasing, 
Note that such a state is reached almost surely in finite time. For $q_N<1$ this follows since the number of edges present in $\bfG^N_n$ is decreasing in time. For $q_N=1$ we obtain a standard voter model on a finite graph which is known to find local consensus.
We denote the final state by
\begin{equation*}
	\bfZ_\infty^N = \lim_{n \to \infty} \bfZ_n^N
\end{equation*}
and the associated population graph by $\bfG^N_\infty$.

Let us recall some basic notions for graphs. Let $(\Lambda,E)$ be an arbitrary graph. We call $x,y\in \Lambda$ \emph{neighbours} if $\{x,y\}\in E$. We say $x$ is \emph{connected} to $y$ if there exists a sequence $x=z_0,\dots, z_n=y$ with $\{z_{i-1},z_i\}\in E$ for all $i\in \{1,\dots, n\}$. We call a graph $(\Lambda,E)$ \emph{connected} if every two vertices $x,y\in \Lambda$ are connected and otherwise we call the graph \emph{disconnected}.

Next, we introduce some notation for relevant quantities of our model. We denote the number of individuals of opinion $1$ by
\begin{equation*}
	Z^N_n = \big\lvert \bfZ_n^N \cap \Lambda^N \big\rvert \in \{0,\ldots, N\}
\end{equation*}
as well as the number of individuals carrying the minority opinion by
\begin{equation*}
	Z_n^{N, \textnormal{min}} = Z^N_n \wedge \big(N - Z^N_n\big).
\end{equation*}
Finally, the minimal degree of $\bfG_n^N$ is denoted by
\begin{equation*}
	D_n^{N, \textnormal{min}} = \min_{x \in \Lambda^N} | \{y \in \Lambda^N : \{x,y\} \in \bfZ_n^N \} | .    
\end{equation*}
There are three possible outcomes for $\bfZ^N_\infty$ described by the following definition.
\begin{definition}[Limit classification]\label{def:LimitClassicifaction}
	\begin{enumerate}
		\item $\mathcal{A}_N \defeq\big\{\bfG_\infty^N$ is disconnected and $Z_\infty^{N, \textnormal{min}}>0\}$ is called the event of \textbf{segregation}, i.e.\ the event that the final graph is disconnected and there is no global consensus.
		\item $\mathcal{B}_N \defeq \big\{\bfG_\infty^N$ is disconnected and $Z_\infty^{N, \textnormal{min}}=0 \big\}$ is the event of \textbf{disconnected consensus}, i.e.\ even though the final graph is disconnected all individuals share the same opinion.
		\item $\mathcal{C}_N \defeq \big\{\bfG_\infty^N$ is connected and $Z_\infty^{N, \textnormal{min}}=0 \big\}$ is the event of \textbf{(connected) consensus}.
	\end{enumerate}
\end{definition}

\begin{remark}
	\begin{enumerate}
		\item $\bfG_\infty^N$ being connected and $Z_\infty^{N, \textnormal{min}}>0$ is not possible, since this would imply that there still exists a discordant edge $\{x,y\}\in \bfE_{\infty}^{N, d}$. However, we already concluded that $\bfE_{\infty}^{N, d}=\emptyset$ must hold. It follows that $\mathcal A_N,\mathcal B_N$ and $\mathcal C_N$ are disjoint and satisfy
		$\IP_N(\mathcal{A}_N \cup \mathcal{B}_N \cup \mathcal{C}_N)=1$.
		\item A disconnected graph can be decomposed into connected components. Thus, even though there is no ``global'' consensus on the event $\cA_N$ the population must have reached a local consensus, which means that every connected component of $\bfG_\infty^N$ is in consensus. This follows by the absence of discordant edges in $\bfG_\infty^N$. 
	\end{enumerate}
\end{remark}

As a last bit of notation we introduce an event of a weaker form of connected consensus. Let $\bfC^{N}_{\max}$ denote the largest connected component of $\bfG_{\infty}^N$. Then, for any $\varepsilon \in \left( 0, \frac{1}{2}\right)$, we define $\mathbf\varepsilon$\emph{-consensus} or \emph{almost-consensus} as
\begin{align*}
	\mathcal{C}^\varepsilon_N \defeq \big\{ &Z_\infty^{N, \textnormal{min}} \leq \varepsilon N  \text{ and } |\bfC^{N}_{\max}|\geq (1-\varepsilon)N \big\}.%\\ &\text{the largest connected component of $\bfG_\infty^N$ is greater than } (1-\varepsilon)N\big\}.
\end{align*}
It holds that $\mathcal{C}_N \subset \mathcal{C}^{\varepsilon_1}_N \subset \mathcal{C}^{\varepsilon_2}_N$ for any $\varepsilon_1 < \varepsilon_2 \in \left(0, \frac{1}{2} \right)$ and $N \in \N$.

Our main focus of this paper is the behaviour of the probabilities of the events $\mathcal{A}_N, \mathcal{B}_N$, $\mathcal{C}^\varepsilon_N$ and $\mathcal{C}_N$ as $N \to \infty$.

\section{Main Results}\label{sec:results}

Our primary interest lies in the scenario when the initial state is a complete graph, where each opinion is equally abundant among the population. This scenario will satisfy the conditions of all our results. However, our methods allow for generalisations in various respective directions.

Before we state our main results, we want to highlight one observation which will be crucial for our proof techniques. 
If one only considers $Z^N=(Z^N_n)_{n\geq 0}$, the total number of individuals of opinion $1$, then $Z^N$ is a delayed simple symmetric random walk until absorption takes place, i.e.\ $\bfE^{N,d}_n=\emptyset$ for the first time. This is a direct consequence of the fact that the ``winning'' opinion in every opinion update is determined by an independent toss of a fair coin.

Note that this is not only the case for the OV-Model, but also for the standard voter model regardless of the underlying graph structure. Hence, consensus of a population of size $N$ described by a standard voter model where initially $\big\lfloor \frac{N}{2} \big\rfloor$ individuals hold opinion $1$, can be identified by the first hitting time $\tau_N$ of the set $\{0, N\}$ of a simple random walk started in $\big\lfloor \frac{N}{2} \big\rfloor$.

Even though this is not the case for the OV-Model, since the underlying population graph can end up being disconnected before the population has reached consensus, the asymptotic behaviour of a simple random walk is still relevant for our first result. Thus, in order to state our first result we cite the following fact about the asymptotic behaviour of the distribution of $\tau_N$:
\begin{equation}\label{eq:spitzer}
	\lim_{N \to \infty} \IP \bigl( \tau_N > N^2 x\bigr) = \frac{4}{\pi} \sum_{j=0}^\infty \frac{(-1)^j}{2j+1}e^{-\frac{\pi^2}{2}(j+1)^2 x} =: \beta(x)
\end{equation}
for any $x>0$. A proof of \eqref{eq:spitzer} can be found in the book of Spitzer \cite{Spitzer1964} (result \textbf{P5} on page 244). Note that it is straightforward to check that the function $\beta : [0, \infty ) \to [0,1]$ is continuous and decreasing in $x$.

Our first result provides an asymptotic lower bound in terms of $\beta$ on the probability of $\cA_N$ as $N\to \infty$, i.e.\ the event of segregation.
\begin{theorem}\label{thm:split}
	Let $q_N = q \in [0,1)$ be constant and let the initial opinions satisfy $Z_0^{N, \textnormal{min}} \geq c N$ with fixed $c \in \left[ 0, \frac{1}{2}\right]$ for all $N$. Then, we get that for any $c' < c$ it holds that
	\begin{align*}
		\liminf_{N \to \infty} \IP \left( Z_{\infty}^{N,\textnormal{min}} > c'N\right) \geq \beta \left( \frac{q}{2(1-q)} \frac{1}{4(c-c')^2}\right) .
	\end{align*}
	This implies that
	\begin{align}
		\liminf_{N \to \infty} \IP \left( \mathcal{A}_N \right) \geq \beta \left( \frac{q}{2(1-q)} \frac{1}{4c^2}\right). \label{eq:split}
	\end{align}
\end{theorem}

This in particular implies that segregation can always take place with positive probability for every constant $q$.
Note that this result is independent of the initial graph structure.

Furthermore, we can show that if $q_N\to 0$, segregation takes place with high probability and the total numbers of both opinions are almost of the same size as initially.
\begin{theorem}\label{thm:split2}
	Let $(q_N)_{N\in \N} \subset (0,1)$ with $q_N\to 0$ and let the initial opinions satisfy $Z_0^{N,\textnormal{min}}\geq cN$ with fixed $c\in[0,\frac12]$ for all $N$. Then, it holds for any $c'<c$ that
	\begin{equation*}
		\lim_{N \to \infty} \mathbb{P} (\mathcal{A}_N) = 1 \quad  \text{and} \quad   \lim_{N \to \infty} \IP \left( Z_{\infty}^{N, \textnormal{min}} > c'N\right) =1
		.
	\end{equation*}
\end{theorem}
This result as well is independent of the initial graph. Similarly to Theorem~\ref{thm:split}, we obtain a lower bound for the probability of $\varepsilon$-consensus in terms of $\beta$.
\begin{proposition}\label{prop:almost_consensus}
	Fix $q_N=q\in[0,1)$ and start the offended voter model with any sequence of initial configurations $\big(\bfZ_0^N\big)_{N \in \N}$ such that the initial graph is the complete graph. Then, for any $0 < \varepsilon < \frac{1}{2}$, 
	\begin{equation*}
		\liminf_{N \to \infty} \IP (\mathcal{C}^\varepsilon_N) \geq 1 - \beta \left(\frac{q}{1-q}\varepsilon (1-\varepsilon)\right) > 0.
	\end{equation*}
\end{proposition}
Taking the limit of $q\to1$ and then $\varepsilon\to0$, this lower bound tends to $1$. Heuristically, since $\mathcal{C}_N^\varepsilon\to \mathcal{C}_N$ as $\varepsilon\to0$, this suggests that consensus occurs with probability tending to $1$ if $q_N\to1$ as $N\to\infty$. Our final result makes this observation rigorous when $q\to1$ fast enough. In particular, it shows then that with high probability consensus is reached and the final graph is densely connected.
\begin{theorem}\label{thm:consensus}
	Let $\kappa \in (0,1)$ and $q_N$ be chosen such that $(1-q_N) \in O(N^{-\delta})$ for some $\delta > 0$. Start the dynamics on the complete graph and arbitrary $Z_0^N \in [0,N]$. Then
	\begin{equation*}
		\lim_{N \to \infty} \IP \bigl( \mathcal{C}_N \text{ and } D^{N, \textnormal{min}}_\infty \geq \kappa N \bigr) = 1 .
	\end{equation*}
\end{theorem}

Note that here and in what follows we use the following big $O$ notation: For real-valued sequences $(a_N)_{N \in \N}$ and $(b_N)_{N \in \N}$ we write $a_N \in O\bigl( b_N\bigr)$ if $\limsup_{N \to \infty} \big|\frac{a_N}{b_N}\big|<\infty$ and $a_N\in \omega(b_N)$ if $\liminf_{N\to\infty} \big|\frac{a_N}{b_N}\big|=\infty$.

\section{Discussion and Simulations}\label{sec:DiscussionAndSimulation}
\subsection*{Comparison of the OV-Model and the voter model with rewiring} As already mentioned, coevolving (also called \emph{adaptive}) network models have rarely been considered in a mathematically rigorous manner. In case of such variants of the voter model it is likely due to loss of the most powerful tool at hand which is the duality relation with coalescing random walks. One of the few works which is also very relevant for us is by Basu and Sly~\cite{basu2017evolving} who consider a voter model with adaptive rewiring. Let us briefly introduce their model before discussing similarities and differences to the offended voter model. 

The initial population structure is given by an \ER graph with $N$ vertices and edge probability $\tfrac{1}{2}$. Again, initially, half of the individuals hold opinion $1$, the other half hold opinion $0$. In this model, in every time step an individual and a neighbour of opposing opinion are chosen uniformly at random. The individual tries to persuade its neighbour which succeeds with probability $q_N=\tfrac{\beta}{N}$. In case this persuasion fails, the first individual cuts the connection and re-connects to a new individual, where multiple edges between the same vertices are allowed. Two cases for this rewiring are considered: either a new individual is chosen uniformly at random among all individuals, which is called \emph{rewire-to-random}, or among all individuals which have the same opinion as the original individual, called \emph{rewire-to-same}. By definition it is clear that this process reaches basically the same classes of final states as the offended voter model.

The authors managed to identify a phase transition in the parameter $\beta$.
Theorem~1(i) in \cite{basu2017evolving} states that in both variants, for $\beta$ small enough, segregation occurs \emph{rapidly} with high probability in the sense that absorption happens after $O(N^2)$ time steps and the frequency of the minority opinion essentially does not move away from $\tfrac{1}{2}N$. On the other hand, part~(ii) of Theorem~1 shows that for arbitrarily small $\varepsilon'>0$, the minority opinions reach frequency $\varepsilon'$ before absorption with high probability, if $\beta$ is large enough (depending on $\varepsilon'$). In this case, the absorption time is of order $N^3$. This might suggest that consensus is possible for large $\beta$, since now there are of order $N^2$ many opinion updates and indeed the frequency of the minority opinion gets arbitrarily close to $0$. However, Theorem~2 contradicts this claim in the \emph{rewire-to-random} case. It shows for this variant that for any $\beta>0$, there is a frequency $\varepsilon_\ast(\beta)>0$ that w.h.p.\ is not reached by the minority opinions before absorption, i.e.\ segregation always happens with high probability. A corresponding result for \emph{rewire-to-same} is not provided.

The fundamental difference between these models is that in adaptive rewiring, edges can be rewired and reused multiple times while in the OV-Model, opinion updates are only allowed to fail once per edge. This leads to a cubic runtime for large $\beta$ in adaptive rewiring and significant movement of the minority opinion. In contrast to that, runtime in the OV-Model cannot exceed quadratic order and our Theorem~\ref{thm:split2} shows, for any $q_N\to0$, that segregation happens rapidly with high probability. Hence, we do not observe a comparable phase transition in the OV-Model.

Basu and Sly only treat the case where $q_N=\tfrac{\beta}{N}$ and did not consider constant $q_N$. However, their results suggest that in this case consensus is reached with high probability. This provides another significant difference to the OV-Model, where for fixed $q>0$ segregation still occurs in the final state with positive probability.
Furthermore, our bound derived in Theorem~\ref{thm:split}, together with simulations depicted in Figure~\ref{fig:bound_split}, suggests that also consensus is reached with positive probability, which makes the outcome of the OV-Model for fixed $q$ non-trivial.

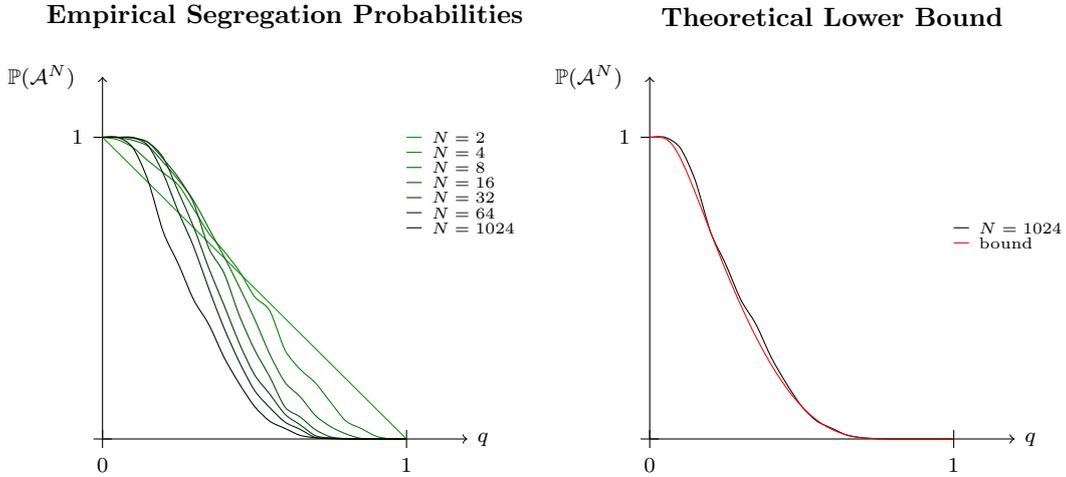
\begin{figure}[htb]
    \centering
    \begin{tikzpicture}[scale=4.0]
    % Draw axes
    \draw[->] (0,0) -- (1.2,0) node[right] {\scriptsize $q$}; % x-axis
    \draw[->] (0,0) -- (0,1.2) node[above] {};

    \draw (1, 0.03) -- (1, -0.03) node[below] {\scriptsize $1$};
    \draw (0, 0.03) -- (0, -0.03) node[below] {\scriptsize $0$};

    \draw (0.03, 1) -- (-0.03, 1) node[left] {\scriptsize $1$};
    \draw (0.03, 0) -- (-0.03, 0) node[left] {};
    
    \draw[black!40!green] (0,1) -- (1,0);
    \draw[black!40!green] (1, 1) -- ++(0.05, 0) node[right] {\tiny  \color{black} $N=2$};
    
    \draw[black!50!green] plot[smooth] file {plots/raw_data/plot_4.table};
    \draw[black!50!green] (1, 0.95) -- ++(0.05, 0) node[right] {\tiny \color{black} $N=4$};
    
    \draw[black!60!green] plot[smooth] file {plots/raw_data/plot_8.table};
    \draw[black!60!green] (1, 0.9) -- ++(0.05, 0) node[right] {\tiny \color{black} $N=8$};

    \draw[black!70!green] plot[smooth] file {plots/raw_data/plot_16.table};
    \draw[black!70!green] (1, 0.85) -- ++(0.05, 0) node[right] {\tiny \color{black} $N=16$};
    
    \draw[black!80!green] plot[smooth] file {plots/raw_data/plot_32.table};
    \draw[black!80!green] (1, 0.8) -- ++(0.05, 0) node[right] {\tiny \color{black} $N=32$};

    \draw[black!90!green] plot[smooth] file {plots/raw_data/plot_64.table};
    \draw[black!90!green] (1, 0.75) -- ++(0.05, 0) node[right] {\tiny \color{black} $N=64$};

    \draw[black] plot[smooth] file {plots/raw_data/plot_1024.table};
    \draw[black] (1, 0.7) -- ++(0.05, 0) node[right] {\tiny \color{black} $N=1024$};

    \node at (0.6, 1.4) {\small\textbf{Empirical Segregation Probabilities}};
    
    \node at (-0.2, 1.2) {\scriptsize $\mathbb{P}(\mathcal{A}^N)$};

    \begin{scope}[xshift=1.8cm]

        % Draw axes
        \draw[->] (0,0) -- (1.2,0) node[right] {\scriptsize $q$}; % x-axis
        \draw[->] (0,0) -- (0,1.2) node[above] {};
    
        \draw (1, 0.03) -- (1, -0.03) node[below] {\scriptsize $1$};
        \draw (0, 0.03) -- (0, -0.03) node[below] {\scriptsize $0$};

        \draw (0.03, 1) -- (-0.03, 1) node[left] {\scriptsize $1$};
        \draw (0.03, 0) -- (-0.03, 0) node[left] {};

        \draw[black] plot[smooth] file {plots/raw_data/plot_1024.table};
        \draw[black] (1, 0.7) -- ++(0.05, 0) node[right] {\tiny \color{black} $N=1024$};
        \draw[color=red] plot[smooth] file {plots/raw_data/spitzer.table};
        \draw[red] (1, 0.65) -- ++(0.05, 0) node[right] {\tiny \color{black} bound};
        \node at (0.6, 1.4) {\small\textbf{Theoretical Lower Bound}};
        \node at (-0.2, 1.2) {\scriptsize $\mathbb{P}(\mathcal{A}^N)$};
    \end{scope}

\end{tikzpicture}
    \caption{Left: Empirical segregation probabilities observed in simulations with $q\in\{0,0.05,0.1,\ldots,1\}$ and varying $N$. For each combination $(q,N)$, $1000$ simulations were carried out. Right: Comparison of the simulated empirical segregation probabilities with $N=1024$ individuals to the lower bound from equation \eqref{eq:split}.}
    \label{fig:bound_split}
\end{figure}

One natural extension of the OV-Model is to allow \emph{befriending}, i.e.\ reconnection of individuals which hold the same opinion, but are disconnected. One possible way to include this is, in every time step, to first assign a fixed weight $\gamma\geq0$ to any missing concordant edge, weight $1$ to every discordant edge and weight $0$ to the remaining edges. Then, randomly sample an edge according to these weights, i.e.\ sample an edge $e$ with probability $w_n(e)/W_n$, where $w_n(e)$ is its individual weight and $W_n=\sum_{f\in E^N}w_n(f)$. If we sample a discordant edge, we proceed as before, and in the other case we reconnect the two individuals. Heuristically, in this variant of the OV-Model we expect that deletions and reconnections balance each other out to some degree since the probability to sample a deleted concordant edge increases when more edges have been deleted and vice versa. Thus, we suspect a similar behaviour as for the model of Basu and Sly in the sense that for fixed $\gamma>0$, consensus is highly probable and behaviour gets more interesting when $\gamma=\gamma_N\to0$.

\subsection*{Segregation and Consensus Probability}

Recall that the total number of individuals of an opinion can be encoded by a symmetric random walk with absorbing boundaries at $0$ and $N$, which we denote by $(Z^{N}_n)_{n\geq 0}$. This walker pauses every time an edge is deleted. Roughly speaking, if the ``opinion random walk'' hits $0$ or $N$ before too many connections are lost, which would cause the graph to become disconnected, then consensus is reached; otherwise the population ends up in segregation. Even though this sounds like a straightforward characterisation, the problem is that the location of edge deletions depends in a highly non-trivial way on opinion evolution.
The locations of edge deletions, in turn, heavily influence the final state of the process: If the deletions concentrate around a single vertex, then the graph disintegrates after $N-1$ deletions. If deletions are "well-mixed", then the graph may be connected until an order of $N^2$ deletions have occurred.

A consequence of \eqref{eq:spitzer} is that the time for $(Z^{N}_n)_{n\geq 0}$ to reach $0$ or $N$ is of order $N^2$ and  with positive probability it can be longer than $cN^2$ for any $c>0$. On the other hand, $\tbinom{N}{2}$ is the maximal number of edges which can be deleted. Thus, the number of deletions until the graph starts to split up into multiple components is at most of order $N^2$. We further believe that the opinions mix sufficiently fast such that deletions occur for a long time in a uniform manner. This would suggest that the time until a component splits off from the rest of the network is also at least of order $N^2$. This heuristics is reflected in our simulations.
Figure~\ref{fig:runtime} suggests existence of a cut-off between runtime of the process until consensus and runtime into an absorbing segregation state. It thus seems that disconnected consensus occurs with probability tending to $0$, leading us to the following conjecture.
Note that in all our simulations with size $N \geq 1000$, we did not observe a single instance of disconnected consensus.
\begin{conjecture}\label{conj:TwoFinalStates}
	Let $q_N = q \in (0,1)$ be constant and $Z_0^{N, \textnormal{min}}= \lfloor \tfrac{N}{2}\rfloor$. Let $E^N\subset \bfZ^N_0$, that is, the initial graph is the complete graph. Then, it holds that 
	\begin{equation*}
		\lim_{N \to \infty} \IP \left( \mathcal{B}_N \right) =0,
	\end{equation*}
	which in particular implies that $\lim_{N \to \infty} \IP \left( \cA_N\cup \cC_N \right) =1$.
\end{conjecture}

\begin{figure}[htb!]
    \centering
    \begin{tikzpicture}
    \begin{axis}[
        width = 6.5cm,
        title = {Time Until Absorption},
        area style,
        axis lines=middle,
        ytick = {0,100,200,300,400},
        yticklabels = {0,100,200,300,400},
        xtick = {748251},
        xticklabels = {\tiny $\frac{1}{1-q}\binom{N}{2}$},
        scaled ticks = false,
        extra y tick labels = {$0$,$\frac{N}{2}$},
        xlabel={$\tau_{\text{abs}}$},
        xlabel style={
            at={(current axis.south)}, 
            anchor=north, 
            yshift=-8pt 
        },
        ylabel={no. of observations},
        ylabel near ticks,
        enlarge x limits = {0.1},
        ymin = 0,
        ymax = 430,
        grid=major,
        legend style={at={(0.3,0.92)},
        anchor=north,draw=none, font=\tiny,
        legend cell align = {left}}
    ]
    \addplot+[ybar interval, color=orange, opacity=0.3] plot coordinates { 
    (680159.,0)
    (691386.,0)
    (702613.,1)
    (713840.,6)
    (725067.,9)
    (736294.,416)
    (747521.,0)
    };
    \addlegendentry{split}

    \addplot+[ybar interval, color=violet, opacity=0.3] plot coordinates { 
    (53339.,  14)
    (119856.1,48)
    (186373.2,69)
    (252890.3,66)
    (319407.4,63)
    (385924.5,65)
    (452441.6,64)
    (518958.7,61)
    (585475.8,66)
    (651992.9,38)
    (718510.,0)
        };
    \addlegendentry{consensus}

    \end{axis}
    \end{tikzpicture}
    \caption{Histogram of $\tau_{\textnormal{abs}}$ with $N=1024$ and $q=0.3$.}
    \label{fig:runtime}
\end{figure}
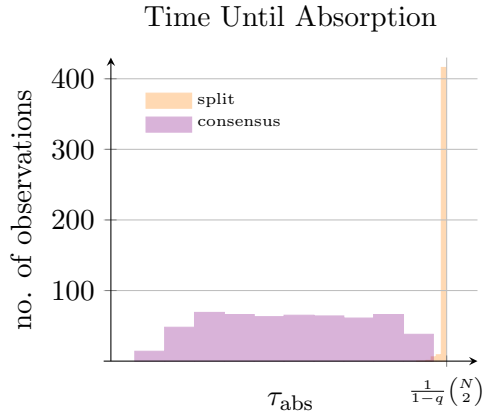

The reasoning above is central to the proof of Theorem~\ref{thm:split}. The strategy is to consider the event that the time until the opinion random walk hits the boundary is longer than it would take to delete all $\tbinom{N}{2}$ edges. The probability of this event is naturally a lower bound on the segregation probability.

In fact, Figure~\ref{fig:runtime} suggests that in case of a split the absorption time concentrates at $\frac{1}{1-q}\tbinom{N}{2}$. In Lemma~\ref{lem:vot_del_moves} it is shown that at this time $\tbinom{N}{2}$ many deletion steps have already occurred with high probability. This leads us to believe that in some cases the lower bound of the segregation probability derived in Theorem~\ref{thm:split} is in fact identical to the limit.
Figure~\ref{fig:bound_split} reinforces this belief since it shows that the lower bound from \eqref{eq:split} is a close fit to the empirical segregation probability for $N\geq1000$. Hence, we have come to the following conjecture.
\begin{conjecture}\label{conj:TightBound}
	Let $q_N = q \in (0,1)$ be constant,  $Z_0^{N, \textnormal{min}}= \lfloor \tfrac{N}{2}\rfloor$ and $E^N\subset \bfZ^N_0$. Then, it holds that 
	\begin{equation*}
		\lim_{N \to \infty} \IP \left( \mathcal{A}_N \right) = \beta \left( \frac{q}{2(1-q)} \right).
	\end{equation*}
\end{conjecture}
Combining the first two conjectures, we strongly believe that for all $q\in(0,1)$ the probability of reaching consensus is strictly positive and the consensus probability converges to $1$ as $q\to 1$. We were unable to show this in full generality since the codependence of edge deletions and the allocation of opinions makes it quite difficult to control the graph evolution as meticulously as necessary to apply our methods.

We were able to show in Theorem~\ref{thm:consensus} that, if $1-q_N\in \mathcal{O}(N^{-\delta})$, then consensus happens with high probability. This is justified by our simulations visualised in Figure~\ref{fig:bound_split}, and thus we believe the following to be true.
\begin{conjecture}
	Let $Z_\infty^{N, \textnormal{min}}>c N$  and $(\Lambda^N,E^N\cap\bfZ_0^{N})$ dense enough. 
	\begin{enumerate}
		\item If $q_N = q \in (0,1)$ is constant, then $\liminf_{N \to \infty} \IP \left( \mathcal{C}_N \right) >0$.
		\item If $q_N \to 1$ as $N\to \infty$, it holds that $\lim_{N \to \infty} \IP \left( \mathcal{C}_N \right)\to 1$.
	\end{enumerate}
\end{conjecture}

Lastly, we strongly suspect monotonicities in $q$ of these probabilities.
\begin{conjecture}
	The probabilities of the final states defined in Definition~\ref{def:LimitClassicifaction} are monotone in $q$, i.e.\  $q\mapsto\IP_q(\cA_N)$ is decreasing and $q\mapsto\IP_q(\cC_N)$ is increasing in $q$.
\end{conjecture}
Intuitively, when $q$ increases, fewer edges will be deleted during the time the random walk takes to reach the boundary and segregation should become less likely. This intuition is reflected in Figure~\ref{fig:bound_split}. However, a proof of this seems elusive since the interdependence of opinion dynamics and deletion locations is difficult to handle.

\subsection*{Structure of the final graph}
As long as the total number of edges is of quadratic order, it seems that the deletion of edges happens fairly uniformly. Heuristically, this can be understood by fixing an edge and noticing that the states of its connected vertices change an order of $N$ times before that edge itself is sampled. Following this argument, the connected vertices of a typical sampled edge should be in a random well-mixed state that does depend on the opinion frequencies but not on the location of the edge.
As a consequence, the induced population graph $\bfG^N_n$ of the OV-Model resembles a dense \ER graph in this stage.
This is in fact one of the key ideas for the proof of Theorem~\ref{thm:consensus}, which we make precise in Section~\ref{sec:ProofConsensus}.
Indeed, this idea is supported by our simulations where in case of consensus even the final graph $\bfG^N_{\infty}$ has edges still of quadratic order.

On the other hand, Figure~\ref{fig:components_split}(b) demonstrates that, in the case of segregation, the number of remaining edges is sub-quadratic and for the most part even of linear order and $\bfG^N_n$ eventually becomes much sparser. Then, the comparability to an \ER graph changes significantly.
One strong indicator for that is that the simulation runs, where segregation is reached as the final state, seem to produce a final graph $\bfG^N_{\infty}$, which contains two connected components of linear size with positive probability. See Figure~\ref{fig:components_split}(a) for a visualisation. In contrast to that, an \ER graph is either fully connected or contains a unique macroscopic component of linear order. 

However, our simulations clearly indicate that in case of segregation the final graph $\bfG_{\infty}^N$ does not always contain two macroscopic components (see Figure~\ref{fig:components_split}(a) and (d)).
It seems that there exists a breaking point with respect to the size of the largest connected component around $\frac45 N$.
If the largest component goes beyond this size, then the second largest component begins to no longer be of linear order.
Intuitively, if the largest component is large, the second component becomes more fragile, a size of linear order becomes more difficult to sustain and it likely breaks apart. This is in line with Figure~\ref{fig:components_split}(d), where the maximum of the observed total number of components coincides with the breaking point in (a). Furthermore, Figures~\ref{fig:components_split}(a) and (c) together show that a significant amount of simulations, with segregation as final state, do produce only one connected component of linear size. This suggests that this is not a finite size effect.

On the other hand, we believe that the total number of the minority opinion is always of linear order if segregation occurs. One can show that this is a consequence of Conjecture~\ref{conj:TightBound} and Theorem~\ref{thm:split}, i.e.\ for $c_N\to0$ it holds that
\begin{equation*}
	\lim_{N \to \infty} \IP \left( c_N N \geq Z_{\infty}^{N, \textnormal{min}} > 0\right)= 0.
\end{equation*}

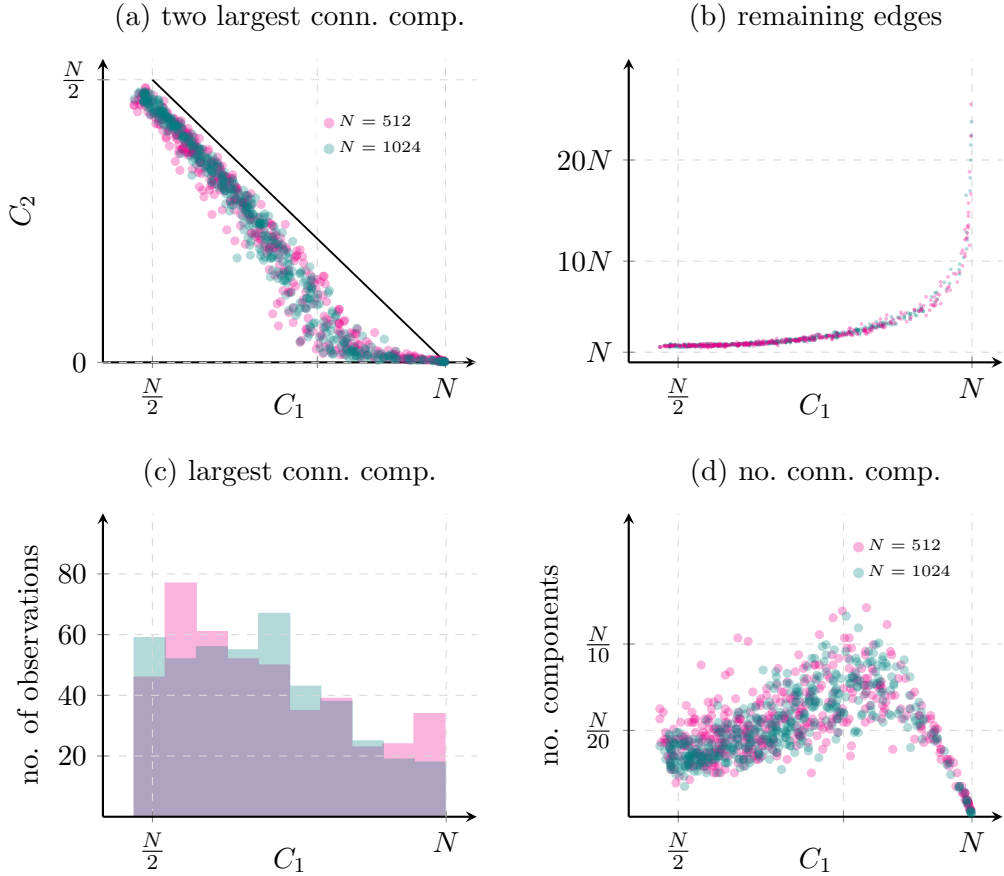
\begin{figure}[htb!]
    \centering
    \begin{tikzpicture}
    \begin{groupplot}[
        group style={group size=2 by 2, vertical sep=2cm, horizontal sep=2cm},
        width=6.5cm,
        axis lines=middle,
        axis line style={thick, black},
        ymin=0,
        ymax=550,
        legend style={
            at={(0.5,1.05)}, anchor=south, draw=none, font=\tiny,
            legend cell align={left}
        },
        legend image post style={line width=1pt, mark repeat=2},
        grid style={dashed,gray!30}
    ]
    
    % --- First Plot ---
    \nextgroupplot[
        title = {(a) two largest conn.\ comp.},
        axis lines=middle,
        xtick = {512, 800},
        ytick = {0},
        extra x ticks = {1024},
        extra y ticks = {0,512},
        xticklabels = {$\frac{N}{2}$, {}},
        extra x tick labels = {$N$},
        extra y tick labels = {$0$,$\frac{N}{2}$},
        xlabel={$C_1$},
        xlabel style={
            at={(current axis.south)}, % Position near the bottom of the axis
            anchor=north, % Align it above the ticks
            yshift=-8pt % Reduce this value to move closer
        },
        ylabel={$C_2$},
        %xlabel near ticks,
        ylabel near ticks,
        enlarge x limits = {0.1},
        ymin = 0,
        ymax = 550,
        grid=major,
        %legend pos=north east, % Adjust the position as needed
        legend style={at={(0.73,0.85)},
        anchor=north,draw=none, font=\tiny,
        legend cell align = {left}},
    ]

        \addplot[only marks, mark=*, color=magenta, opacity=0.3, mark size=1.5 pt] table [x expr=\thisrowno{0}*2, y expr=\thisrowno{1}*2] {plots/raw_data/components_512};
        \addlegendentry{$N=512$}

        \addplot[only marks, mark=*, color=teal, opacity=0.3, mark size=1.5 pt] table {plots/raw_data/components_1024};
        \addlegendentry{$N=1024$}

        \addplot[color=black, line width=0.7pt] coordinates {(512,512) (1024,0)};
    % --- Second Plot (Below) ---
    \nextgroupplot[
        title = {(b) remaining edges},
        axis lines=middle,
        %xlabel near ticks,
        ylabel near ticks,
        ytick = {1,10,20},
        yticklabels = {$N$,$10N$,$20N$},
        xtick = {0.5,1},
        xticklabels = {$\frac{N}{2}$,$N$},
        xlabel={$C_1$},
        xlabel style={
            at={(current axis.south)}, % Position near the bottom of the axis
            anchor=north, % Align it above the ticks
            yshift=-8pt % Reduce this value to move closer
        },
        enlarge x limits = {0.1},
        ymin = 0,
        ymax = 30,
        grid=major,
        %legend pos=north east, % Adjust the position as needed
        legend style={at={(0.73,0.85)},
        anchor=north,draw=none, font=\tiny,
        legend cell align = {left}},
    ]

        \addplot[only marks, mark=*, color=teal, opacity=0.3, mark size=0.5 pt] table {plots/raw_data/remaining_edges_1024};

        \addplot[only marks, mark=*, color=magenta, opacity=0.3, mark size=0.5 pt] table {plots/raw_data/remaining_edges_512};
%        \addlegendentry{$N=1024$}

%        \addlegendentry{$N=1024$}

    \nextgroupplot[
        title = {(c) largest conn.\ comp.},
        area style,
        axis lines=middle,
        xtick = {512},
        ytick = {0,20,40,60,80},
        extra x ticks = {1024},
        xticklabels = {$\frac{N}{2}$},
        extra x tick labels = {$N$},
        extra y tick labels = {$0$,$\frac{N}{2}$},
        xlabel={$C_1$},
        xlabel style={
            at={(current axis.south)}, % Position near the bottom of the axis
            anchor=north, % Align it above the ticks
            yshift=-8pt % Reduce this value to move closer
        },
        ylabel={no. of observations},
        ylabel near ticks,
        enlarge x limits = {0.1},
        ymin = 0,
        ymax = 100,
        grid=major,
        %legend pos=north east, % Adjust the position as needed
        legend style={at={(0.73,0.85)},
        anchor=north,draw=none, font=\tiny,
        legend cell align = {left}}
    ]
    \addplot+[ybar interval, color=magenta, opacity=0.3] plot coordinates { ( 480.0 , 46.0 )
( 534.2 , 77.0 )
( 588.4 , 61.0 )
( 642.6 , 52.0 )
( 696.8 , 50.0 )
( 751.0 , 35.0 )
( 805.2 , 39.0 )
( 859.4 , 23.0 )
( 913.6 , 24.0 )
( 967.8 , 34.0 )
(1022, 0)};

    \addplot+[ybar interval, color=teal, opacity=0.3] plot coordinates { ( 480.0 , 59.0 )
( 534.3 , 52.0 )
( 588.6 , 56.0 )
( 642.9 , 55.0 )
( 697.2 , 67.0 )
( 751.5 , 43.0 )
( 805.8 , 38.0 )
( 860.1 , 25.0 )
( 914.4 , 19.0 )
( 968.7 , 18.0 )
( 1023 , 0 )}; 

    \addplot[only marks, mark=*, color=teal, opacity=0, mark size=1.5 pt] table {plots/raw_data/components_1024};

    \addplot[only marks, mark=*, color=magenta, opacity=0, mark size=1.5 pt] table [x expr=\thisrowno{0}*2, y expr=\thisrowno{1}*2] {plots/raw_data/components_512};

    \nextgroupplot[
        title = {(d) no.\ conn.\ comp.},
        axis lines=middle,
        ytick = {51.2,102.4},
        xtick = {512,800,1024},
        xticklabels = {$\frac{N}{2}$, {}, $N$},
        yticklabels = {$\frac{N}{20}$,$\frac{N}{10}$},
        extra x tick labels = {$N$},
        extra y tick labels = {$0$,$\frac{N}{2}$},
        xlabel={$C_1$},
        xlabel style={
            at={(current axis.south)}, % Position near the bottom of the axis
            anchor=north, % Align it above the ticks
            yshift=-8pt % Reduce this value to move closer
        },
        ylabel={no. components},
        %xlabel near ticks,
        ylabel near ticks,
        enlarge x limits = {0.1},
        ymin = 0,
        ymax = 180,
        grid=major,
        %legend pos=north east, % Adjust the position as needed
        legend style={at={(0.74,0.95)},
        anchor=north,draw=none, font=\tiny,
        legend cell align = {left}},
    ]

        \addplot[only marks, mark=*, color=magenta, opacity=0.3, mark size=1.5 pt] table [x expr=\thisrowno{0}*2, y expr=\thisrowno{1}*2] {plots/raw_data/num_components_512};
        \addlegendentry{$N=512$}

        \addplot[only marks, mark=*, color=teal, opacity=0.3, mark size=1.5 pt] table {plots/raw_data/num_components_1024};
        \addlegendentry{$N=1024$}

    \end{groupplot}
\end{tikzpicture}
    \caption{
    Segregation cases of 1000 simulations for each combination of $q=0.3$ and $N\in\{512,1024\}$. \textbf{(a)} Size $C_1$ of largest versus size $C_2$ of second largest connected component. \textbf{(b)} Number of remaining edges in final graph. \textbf{(c)} Histogram of $C_1$. \textbf{(d)} Number of connected components in final graph.
    }
    \label{fig:components_split}
\end{figure}

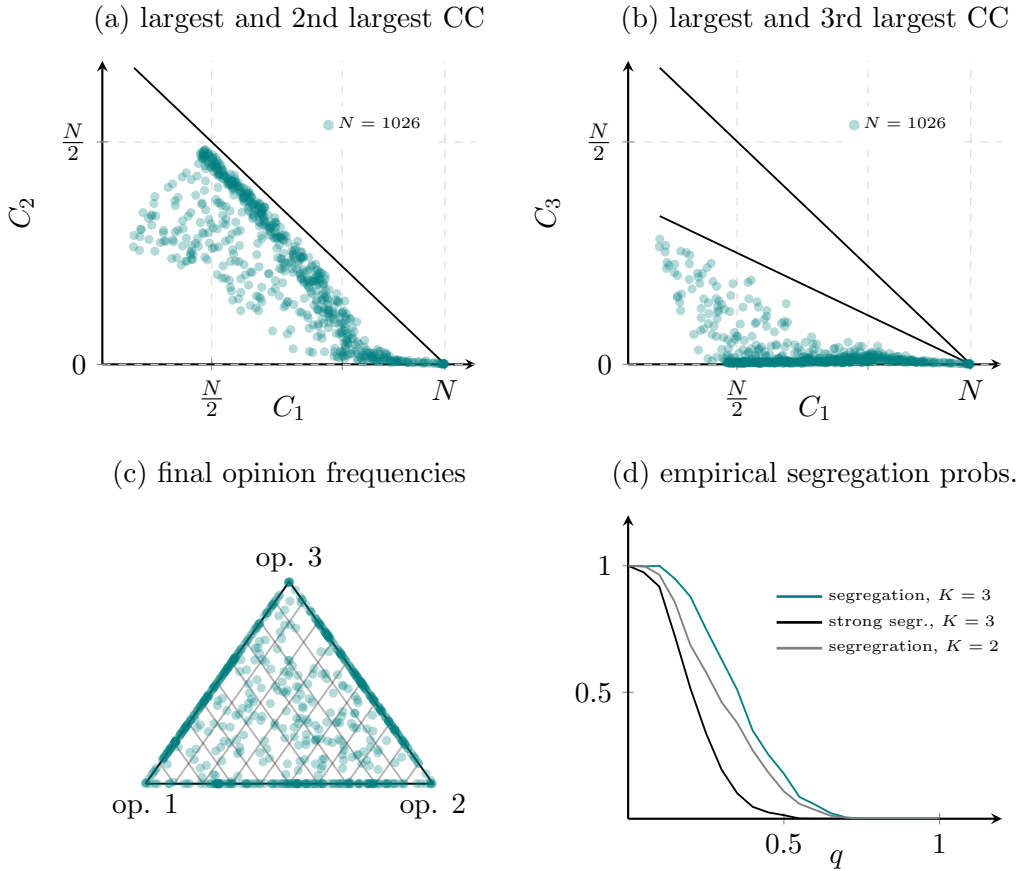
\begin{figure}[htb!]
    \centering
    \begin{tikzpicture}
    \begin{groupplot}[
        group style={group size=2 by 2, vertical sep=2cm, horizontal sep=2cm},
        width=6.5cm,
        axis lines=middle,
        axis line style={thick, black},
        ymin=0,
        ymax=550,
        legend style={
            at={(0.5,1.05)}, anchor=south, draw=none, font=\tiny,
            legend cell align={left}
        },
        legend image post style={line width=1pt, mark repeat=2},
        grid style={dashed,gray!30}
    ]
    
    % --- First Plot ---
    \nextgroupplot[
        title = {(a) largest and 2nd largest CC},
        axis lines=middle,
        xtick = {513, 800},
        ytick = {0},
        extra x ticks = {1026},
        extra y ticks = {0,513},
        xticklabels = {$\frac{N}{2}$, {}},
        extra x tick labels = {$N$},
        extra y tick labels = {$0$,$\frac{N}{2}$},
        xlabel={$C_1$},
        xlabel style={
            at={(current axis.south)}, % Position near the bottom of the axis
            anchor=north, % Align it above the ticks
            yshift=-8pt % Reduce this value to move closer
        },
        ylabel={$C_2$},
        %xlabel near ticks,
        ylabel near ticks,
        enlarge x limits = {0.1},
        ymin = 0,
        ymax = 700,
        grid=major,
        %legend pos=north east, % Adjust the position as needed
        legend style={at={(0.73,0.85)},
        anchor=north,draw=none, font=\tiny,
        legend cell align = {left}},
    ]

        \addplot[only marks, mark=*, color=teal, opacity=0.3, mark size=1.5 pt] table {plots/three_opinions/connected_components_1_2_1026};
        \addlegendentry{$N=1026$}

        \addplot[color=black, line width=0.7pt] coordinates {(342,684) (1026,0)};

    % --- Second Plot (Below) ---
    \nextgroupplot[
        title = {(b) largest and 3rd largest CC},
        axis lines=middle,
        xtick = {512, 800},
        ytick = {0},
        extra x ticks = {1026},
        extra y ticks = {0,513},
        xticklabels = {$\frac{N}{2}$, {}},
        extra x tick labels = {$N$},
        extra y tick labels = {$0$,$\frac{N}{2}$},
        xlabel={$C_1$},
        xlabel style={
            at={(current axis.south)}, % Position near the bottom of the axis
            anchor=north, % Align it above the ticks
            yshift=-8pt % Reduce this value to move closer
        },
        ylabel={$C_3$},
        %xlabel near ticks,
        ylabel near ticks,
        enlarge x limits = {0.1},
        ymin = 0,
        ymax = 700,
        grid=major,
        %legend pos=north east, % Adjust the position as needed
        legend style={at={(0.73,0.85)},
        anchor=north,draw=none, font=\tiny,
        legend cell align = {left}},
    ]

        \addplot[only marks, mark=*, color=teal, opacity=0.3, mark size=1.5 pt] table {plots/three_opinions/connected_components_1_3_1026};
        \addlegendentry{$N=1026$}

        \addplot[color=black, line width=0.7pt] coordinates {(342,684) (1026,0)};
        \addplot[color=black, line width=0.7pt] coordinates {(342,342) (1026,0)};
       
    \nextgroupplot[
        title = {(c) final opinion frequencies},
        axis line style=transparent,
        tick style = transparent,
        tick label style = transparent,
        xlabel style = transparent,
        ylabel style = transparent,
        xtick = {512, 800},
        ytick = {0},
        extra x ticks = {1026},
        extra y ticks = {0,1026},
        xticklabels = {$\frac{N}{2}$, {}},
        extra x tick labels = {$N$},
        extra y tick labels = {$0$,$\frac{N}{2}$},
        xlabel={$C_1$},
        %ylabel={$C_2$},
        %xlabel near ticks,
        %ylabel near ticks,
        enlarge x limits = {0.15},
        enlarge y limits = {0.15},
        ymin = 0,
        ymax = 1026,
        xmin = 0,
        xmax = 1026,
        grid=none,
        %legend pos=north east, % Adjust the position as needed
        legend style={at={(0.73,0.85)},
        anchor=north,draw=none, font=\tiny,
        legend cell align = {left}},
    ]

        \addplot[color=black, line width=0.7pt, opacity=0.3] coordinates {
        (0.00, 0.0)
        (51.3, 88.85420642829999)
        (102.6, 0)
        (205.2, 0)
        (102.6, 177.70841285659998)
        (153.89999999999998, 266.5626192849)
        (307.79999999999995, 0)
        (410.4,0)
        (205.2, 355.41682571319996)
        (256.5, 444.2710321414999)
        (513.0, 0)
        (615.5999999,0)
        (307.79999999999995, 533.1252385698)
        (359.09999999999997, 621.9794449981)
        (718.1999999999999,0)
        (820.8,0)
        (410.4, 710.8336514263999)
        (461.7, 799.6878578546999)
        (923.4,0)};

        \addplot[color=black, line width=0.7pt, opacity=0.3] coordinates {
        (1026.0, 0.0)
        (923.4,0)
        (974.7, 88.85420642829999)
        (923.4, 177.70841285659998)
        (820.8, 0)
        (718.2, 0)
        (872.1, 266.5626192849)
        (820.8, 355.41682571319996)
        (615.6, 0)
        (513.0, 0)
        (769.5, 444.2710321414999)
        (718.2, 533.1252385698)
        (410.4, 0)
        (307.8, 0)
        (666.9, 621.9794449981)
        (615.6, 710.8336514263999)
        (205.2, 0)
        (102.6, 0)
        (564.3, 799.6878578546999)
        (513.0, 888.5420642829998)
        (0, 0)
        };

        \addplot[color=black, line width=0.7pt] coordinates {(0,0) (1026,0) (513,888) (0,0)};

        \addplot[only marks, mark=*, color=teal, opacity=0.3, mark size=1.5 pt] table[x expr=\thisrowno{1} + 0.5* (1026- \thisrowno{0}-\thisrowno{1}), y expr=0.866 * (1026- \thisrowno{0}-\thisrowno{1})] {plots/three_opinions/final_opinions_1_2_1026};

        \node[below] at (0,0) {op. 1};
        \node[below] at (1026,0) {op. 2};
        \node[above] at (513,888.542) {op. 3};

    \nextgroupplot[xmax=1.2, ymax=1.2,
    legend style={at={(0.7,0.5)}},
    title={(d) empirical segregation probs.},
    xlabel={$q$},
        xlabel style={
            at={(current axis.south)}, % Position near the bottom of the axis
            anchor=north, % Align it above the ticks
            yshift=-8pt,
            xshift=8pt
        },
    ylabel={},
        ylabel style={
            at={(current axis.north)}, % Position near the bottom of the axis
            anchor=north, % Align it above the ticks
            yshift=0pt,
            xshift=-90pt
        },
    ]
        \addplot[color=teal, line width=0.7pt] coordinates {
        ( 0.0 ,  1.0 )
        ( 0.05 ,  0.997 )
        ( 0.1 ,  1.0 )
        ( 0.15000000000000002 ,  0.9490000000000001 )
        ( 0.2 ,  0.878 )
        ( 0.25 ,  0.75 )
        ( 0.30000000000000004 ,  0.63 )
        ( 0.35000000000000003 ,  0.51 )
        ( 0.4 ,  0.34800000000000003 )
        ( 0.45 ,  0.254 )
        ( 0.5 ,  0.178 )
        ( 0.55 ,  0.08600000000000001 )
        ( 0.6000000000000001 ,  0.056 )
        ( 0.65 ,  0.022 )
        ( 0.7000000000000001 ,  0.005 )
        ( 0.75 ,  0.0 )
        ( 0.8 ,  0.0 )
        ( 0.8500000000000001 ,  0.0 )
        ( 0.9 ,  0.0 )
        ( 0.9500000000000001 ,  0.0 )
        ( 1.0 ,  0.0 )
        };
        \addlegendentry{segregation, $K=3$}

        \addplot[color=black, line width=0.7pt] coordinates {
        ( 0.0 ,  1.0 )
        ( 0.05,  0.974)
        ( 0.1 ,  0.918)
        ( 0.15,  0.722)
        ( 0.2 ,  0.515)
        ( 0.25,  0.338)
        ( 0.30,  0.194)
        ( 0.35,  0.101)
        ( 0.4 ,  0.047)
        ( 0.45,  0.025)
        ( 0.5 ,  0.014)
        ( 0.55,  0.001)
        ( 0.60,  0)
        ( 0.65,  0)
        ( 0.70,  0)
        ( 0.75,  0)
        ( 0.8 ,  0)
        ( 0.85,  0)
        ( 0.9 ,  0)
        ( 0.95,  0)
        ( 1.0 ,  0)
        };
        \addlegendentry{strong segr., $K=3$}

        \addplot[color=gray, line width=0.7pt] coordinates {
        ( 0.0 ,  1.0 )
        ( 0.05 ,  1.0 )
        ( 0.1 ,  0.965 )
        ( 0.15000000000000002 ,  0.856 )
        ( 0.2 ,  0.6859999999999999 )
        ( 0.25 ,  0.579 )
        ( 0.30000000000000004 ,  0.46099999999999997 )
        ( 0.35000000000000003 ,  0.379 )
        ( 0.4 ,  0.27 )
        ( 0.45 ,  0.18300000000000005 )
        ( 0.5 ,  0.10799999999999998 )
        ( 0.55 ,  0.05900000000000005 )
        ( 0.6000000000000001 ,  0.03500000000000003 )
        ( 0.65 ,  0.01200000000000001 )
        ( 0.7000000000000001 ,  0.0040000000000000036 )
        ( 0.75 ,  0.0010000000000000009 )
        ( 0.8 ,  0.0 )
        ( 0.8500000000000001 ,  0.0 )
        ( 0.9 ,  0.0 )
        ( 0.9500000000000001 ,  0.0 )
        ( 1.0 ,  0.0 )
        };
        \addlegendentry{segregation, $K=2$}

    \end{groupplot}
\end{tikzpicture}
    \caption{
   Simulation of three opinions. Depiction of 633 segregation states reached in 1000 runs of $N=1026$, $q=0.3$, initial complete graph with $342$ nodes of each opinion.
   \textbf{(a)} Sizes of largest versus second largest connected component.
   \textbf{(b)} Largest versus third largest connected component.
   \textbf{(c)} Ternary plot of opinion frequencies at absorption.
   \textbf{(d)} Comparison of empirical probabilities for $q\in\{0,0.05,0.1,\ldots1\}$ and $K=3$. Teal: Any segregation. Black: Segregation with 3 opinions present. Gray: Segregation with $K=2$ initial opinions (cf.\ Figure~\ref{fig:bound_split} for $N=1024$).
    }\label{fig:three-opinions}
\end{figure}

\subsection*{Multiple opinions}
A natural generalisation of the (two-opinion) OV-Model is the consideration of $K\in\N$, $K\geq2$, opinions present in the population. Letting $[K]=\{1,\ldots,K\}$ be the set of opinions, the formal extension works as follows:

Consider a graph $\mathbf{G}_n^N = (\Lambda^N, \mathbf{E}_n^N)$ where $\mathbf{E}_n^N$ is a (randomly evolving) edge set.
Define $\mathbf{Z}_n^N : \Lambda^N \to [K]$ to be the (randomly evolving) opinion function.
Denote the set of discordant edges by
\begin{equation*}
	\mathbf{E}_n^{N, d} := \{ \{i, j\} \in \mathbf{E}_n^N : \mathbf{Z}_n^N (i) \neq \mathbf{Z}_n^N (j)\}.
\end{equation*}
Start the dynamics with arbitrary initial condition $\mathbf{E}_0^N$ and $\mathbf{Z}_0^N$. Then, at each time step $n \in \mathbb{N}$, sample a discordant edge $\{i, j\} \in \mathbf{E}_{n-1}^{N,d}$ uniformly at random and
\begin{itemize}
	\item with probability $\frac{q}{2}$, set $\mathbf{Z}_{n}^N(i) = \mathbf{Z}_{n-1}^N (j)$
	\item with probability $\frac{q}{2}$, set $\mathbf{Z}_{n}^N(j) = \mathbf{Z}_{n-1}^N (i)$
	\item with probability $1-q$, set $\mathbf{E}_n^N = \mathbf{E}_{n-1}^N \setminus \big\{\{i,j\}\big\}$
\end{itemize}
while all other quantities remain constant.

As a first observation note that the number of opinion-$k$-individuals, $k\in K$, now does not necessarily change with every opinion update, if $K\geq3$.
Intuitively, this introduces an additional delay for the maximum opinion random walk which makes consensus harder to achieve and thus segregation more likely. Indeed, this is reflected in Figure~\ref{fig:three-opinions}(d).

Several interesting questions emerge for this generalisation. For example, is it possible for any $K$ that a strong form of segregation occurs with positive probability, i.e.\ all opinions are still present and of macroscopic order in the final state, or is only a portion of the opinions asymptotically persistent. Furthermore, it would be interesting to know what the size of the connected components in case of segregation is. Are there multiple components of macroscopic order? There are many more questions of this flavour. Figure~\ref{fig:three-opinions} provides first insights: (d) shows that, for $q=0.3$, around a third of the segregation cases are strong (black and teal lines), which correspond to the dots in the interior of (c). Since for this $q$ a large proportion of runs end up with only two opinions left, (a) shows similar effects as Figure~\ref{fig:components_split}(a), with added noise from the strong segregation cases.

In fact, we can show with the methods of this paper that strong segregation occurs with positive probability for $q$ small enough.
To this end, denote the number of individuals carrying opinion $k \in [K]$ by
\begin{equation*}
	Z_n^{N, k} = |\{i \in \Lambda^N : \mathbf{Z}_n^N (i) = k \}| .
\end{equation*}
Note that $\big(Z_n^{N,k}\big)_{n \in \mathbb{N}}$ is again for all $k$ a delayed simple symmetric random walk until $\mathbf{E}_n^{N,d} = \emptyset$ for the first time. In the case $K=2$, the delay is particularly well behaved since the delay corresponds to the edge deletions. This delay is studied in detail in the proofs of Theorem~\ref{thm:split} and Lemma~\ref{lem:vot_del_moves}.

For $K \geq 3$ the delay becomes more complex since $Z_n^{N,k}$ only jumps if an opinion update which involves an individual of opinion $k$ is performed. Thus, the delay is highly non-trivial as it depends on the precise allocation of the opinions. Despite the more complicated nature of this twofold delay, we still obtain a result in the spirit of Theorem~\ref{thm:split} for arbitrarily many opinions.

To this end, we generalize the notion of ``minority opinion'' to the case of $K$ possible opinions, i.e.
\begin{equation*}
	Z_n^{N, \textnormal{min}} := \min_{k \in [K]} Z_n^{N, k}
\end{equation*}
\begin{proposition}\label{prop:more_opinions}
	Let $K \in \N$ be the number of opinions and $\delta \in [0, K^{-1})$. Start the OV-Model such that there are at least $\lfloor K^{-1}N\rfloor$ voters of each opinion. Then there exists $q^*_K>0$ such that for any $q \leq q_K^*$
	\begin{equation*}
		\liminf_{N \to \infty} \IP_{K,q} (Z_\infty^{N, \textnormal{min}}>\delta N)>0 .
	\end{equation*}
\end{proposition}
A sketch of the proof is given in Subsection~\ref{sec:ProofsSegregation}. It is possible to compute $q_K^*$ explicitly, but we do not provide values for $q_K^*$, since we have no reason to doubt that Proposition~\ref{prop:more_opinions} holds for arbitrary $q \in [0,1)$ and not only $q \leq q_K^*$.
The appearance of a $q_K^*$ is rather an artefact of the proof strategy, which comes from a lack of control over the system of (dependent) delayed random walks.

\section{Construction}\label{sec:construction}
In this section we provide two explicit constructions of the OV-Model which will be useful in Section~\ref{sec:proofs}.

In the first construction of the OV-Model we sample in every step a \emph{discordant} edge after which either the sampled edge is deleted or an opinion update takes place. Alternatively, as in the second construction, one can also sample an \emph{arbitrary pair of vertices} in every step. If the corresponding edge is still present and discordant, one again performs either a deletion or an opinion update, and if it is not, nothing happens. Thus, this results in a (time) \textit{delayed OV-Model}. We define this version and couple it with a third model which we call the \textit{dynamical deletion graph}.

\paragraph*{OV-Model}
Let $N, \Lambda^N, E^N$ and $G^N$ be as in Section~\ref{sec:model_and_notation} and fix $q_N\in[0,1]$.
Let $\bfZ^N$ be a Markov chain taking values in $\cP (\Lambda^N \cup E^N)$ and let $\bfE_n^{N, d}$ be as defined in Section~\ref{sec:model_and_notation}.
An alternative way to state the dynamics described in Section \ref{sec:model_and_notation} is as follows:
Let $(U^N_n)_{n \in \N}$ and $(V^N_n)_{n \in \N}$ be two independent sequences of i.i.d.\ random variables such that $U^N_1\sim \text{Ber}(q_N)$ and $V^N_1\sim \text{Ber}\bigl(\tfrac{1}{2}\bigr)$.
Start the process with arbitrary initial state $\bfZ^N_0$. If at time $n=1,2,\ldots$ $\bfE^{N,d}_{n-1}\neq\emptyset$, choose a discordant edge $\{x,y\} \in \bfE_{n-1}^{N, d}$ uniformly at random and independently of everything else. Then,
\begin{itemize}
	\item if $U_n^N = 0$, remove $\{x,y\}$, i.e.\ set $\bfZ_{n}^N = \bfZ_{n-1}^N \setminus \{\{x,y\}\}$, and
	\item if $U_n^N = 1$, then
	\begin{itemize}
		\item if $V_n^N = 1$, set  $\bfZ_n^N = \bfZ_{n-1}^N \cup \{x,y\}$ and \vspace{.5em}
		\item if $V_n^N = 0$, set  $\bfZ_n^N = \bfZ_{n-1}^N \setminus \{x,y\}$.
	\end{itemize}
\end{itemize}
On the other hand, if $\bfE^{N,d}_{n-1}=\emptyset$, set $\bfZ^N_n=\bfZ^N_{n-1}$. We denote the time of absorption by
\begin{equation*}
	\tau_{\textnormal{abs}} = \inf \{ k \in \N : \bfE^{N,d}_k = \emptyset \}.
\end{equation*}
With this construction, some quantities of the process can be nicely expressed in terms of $U^N$ and $V^N$. Specifically, for all $n\leq \tau_{\textnormal{abs}}$ \textit{the number of individuals having opinion $1$} at time  $n$, is given by
\begin{equation*}
	Z_n^N = Z_0^N + \sum_{i=1}^n U_i^N \bigr( 2 V_i^N - 1\bigl)
\end{equation*}
and the \textit{number of opinion updates} until time $n$ is given by $S_n^{N, \textnormal{op}} = \sum_{i=1}^n U_i^N$. Likewise, the number of \textit{edge deletions} until time $n$ is $S_n^{N, \textnormal{del}} = n - S_n^{N, \textnormal{op}}$.

\begin{remark}\label{rem:ExtensionBeyondAbsorption}
	Note that we abuse notation and define $Z^N, S^{N,\textnormal{op}}$ and $S^{N,\textnormal{del}}$ for all times $n\in\N$, even though $\mathbf Z^N$ stays constant for $n\geq\tau_{\textnormal{abs}}$. This is feasible since these processes as described above solely depend on the continuing sequences $U^N$ and $V^N$. As we will see in the proofs to come, it will be convenient to have $Z^N, S^{N, \textnormal{op}}, S^{N, \textnormal{del}}$ keep evolving after absorption of $\bfZ^N$. Notably, after $\tau_{\textnormal{abs}}$ they lose their interpretation as opinions or moves in the OV-Model.
\end{remark}

\paragraph*{Delayed OV-Model and Coupling with the Dynamical Deletion %(\ER)
	Graph}
Together with the delayed OV-Model, which is denoted by $\bfhZ^{N}=(\bfhZ^{N}_n)_{n \in \N_0}$, we define a dynamical random graph $(\Lambda^N, \cE^N)$ where the stochastic process $\cE^N=(\cE^N_n)_{n\in \N_0}$ models its edge set which evolves randomly over time. We call $(\Lambda^N, \cE^N)$ the \emph{dynamical deletion graph}.
The goal of the following construction is to obtain a coupling between these processes such that at all times $n$ the graph $(\Lambda^N, \cE^N_n)$ is a subgraph of $\bfhG^N_n=(\Lambda^N, \bfhZ^N_n\cap E^N)$.
This coupling and further auxiliary processes obtained from this construction will be central to later proofs.

Let $(\cE_n^N)_{n \in \N_0}$ and $(\bfhZ^{N}_n)_{n \in \N_0}$ be discrete-time Markov chains taking values in $\cP (E^N)$ and $\cP (\Lambda^N \cup E^N)$, respectively.
Initialise both processes with the same edge sets, i.e. $\cE_0^N = \bfhZ^{N}_0 \cap E^N$.
Denote the set of all discordant edges of the process $\bfhZ^{N}$ at time $n$ by

\begin{equation*}
	\bfhE_n^{N, d} \defeq \Big\{ \{x,y\} \in \bfhZ^{N}_n : \big|\bfhZ^N_n \cap \{x,y\}\big|=1\Big\}.
\end{equation*}
Analogously to the non-delayed construction, let $(\hU^N_n)_{n \in \N}$ and $(\hV^N_n)_{n \in \N}$ be two independent sequences of i.i.d.\ random variables such that $\hU^N_1\sim \text{Ber}(q_N)$ and $\hV^N_1\sim \text{Ber}(\tfrac{1}{2})$.
The auxiliary processes $(X^{N,\textnormal{op}}_n)_{n\in\N}$ and $(X^{N,\textnormal{e}}_n)_{n\in\N}$ defined in the dynamics below will become useful later (cf.\ Remark~\ref{rem:delayed_ov_model}).

At time $n =1,2,\ldots$ choose any pair of vertices $\{x,y\} \in E^N$ uniformly at random independently of everything else.
\begin{itemize}
	\item If $\hU_n^N = 0$, set $X^{N, \textnormal{op}}_n = 0$ and set $\cE^N_{n} = \cE^N_{n-1} \setminus \{\{x,y\}\}$, i.e.\ remove $\{x,y\}$ from $\cE^N$ if it is still present. Also,
	\begin{itemize}
		\item if $\{x,y\} \in \bfhE_{n-1}^{N, d}$, then remove it from $\bfhZ^{N}$ as well, i.e.\ set $\bfhZ^{N}_n = \bfhZ^{N}_{n-1} \setminus \{\{x,y\}\}$ and $X_n^{N,\textnormal{e}} = 1$ or\vspace{.5em}
		\item if $\{x,y\} \notin \bfhE_{n-1}^{N, d}$, then set $\bfhZ^{N}_n = \bfhZ^{N}_{n-1}$ and set $X_n^{N,\textnormal{e}} = 0$.
	\end{itemize}
	\item If $\hU_n^N = 1$, then set $\cE^N_{n} = \cE^N_{n-1}$ and
	\begin{itemize}
		\item if $\{x,y\} \in \bfhE_{n-1}^{N, d}$, set $X^{N, \textnormal{e}}_n = 1$ and $X^{N, \textnormal{op}}_n = 1$ and further
		\begin{itemize}
			\item if $\hV_n^N = 1$, set $\bfhZ^{N}_{n} = \bfhZ^{N}_{n-1} \cup \{x,y\}$ or\vspace{.5em}
			\item if $\hV_n^N = 0$, set $\bfhZ^{N}_{n} = \bfhZ^{N}_{n-1} \setminus \{x,y\}$,\vspace{.5em}
		\end{itemize}
		\item if $\{x,y\} \notin \bfhE_{n-1}^{N, d}$ set $\bfhZ^{N}_{n} = \bfhZ^{N}_{n-1}$ and set $X^{N, \textnormal{e}}_n = 0$ and $X^{N, \textnormal{op}}_n = 0$.
	\end{itemize}
\end{itemize}
Note that in the original OV-Model, in every time step a random \textit{discordant} edge is sampled, while in the delayed OV-Model, which we just defined, \textit{any} random pair of vertices $\{x,y\} \in E$ is sampled. Thus, it becomes apparent that it is just a time delayed version of the former process. In fact the jump chain of the delayed OV-Model $\bfhZ^N$ has the same distribution as the OV-Model $\bfZ^N$. 

From the definition of the two processes it is clear that $(\Lambda^N, \cE^N_n)$ is a subgraph of $(\Lambda^N, \bfhZ^{N}_n \cap E^N)$ at all times $n \in \N_0$.
This will allow us to study the connectedness of the delayed OV-Model by understanding the connectedness of $\cE^N$ in certain regimes, which can be achieved via comparison to \ER graphs.

\begin{remark}\label{rem:delayed_ov_model}
	From the construction it becomes clear that $X^{N, \textnormal{e}}_n=1$ whenever a change in the delayed OV-Model occurs and $X^{N, \textnormal{op}}_n = 1$ whenever an opinion update takes place. In fact, it is straightforward to see that
	\begin{equation*}
		\IP (X_n^{N,\textnormal{e}} = 1 | \bfhZ_{n-1}^N) = \frac{2 \hE_{n-1}^{N, d}}{N(N-1)}\quad \text{and} \quad \IP (X_n^{N, \textnormal{op}} = 1 | \bfhZ_{n-1}^N) = \frac{2 q_N \hE_{n-1}^{N, d}}{N(N-1)},
	\end{equation*}
	where $\hE_n^{N,d} := |\bfhE_{n}^{N, d}|$ denotes the number of discordant edges at time $n$.
	We will use this property for the proofs in Section~\ref{sec:proofs}.
\end{remark} 

We end this section introducing some further notation for the delayed OV-Model, which we will need later. Analogously to the non-delayed case we denote the minimal degree of the graph $\bfhG_n^N$ of the delayed OV-Model at time $n$ by 
\begin{equation*}
	\hD_n^{N, \textnormal{min}} = \min_{x \in \Lambda^N} | \{y \in \Lambda^N : \{x,y\} \in \bfhZ_n^N \} | .
\end{equation*}
The number of individuals holding opinion one is denoted by $\hZ^N_n$ and the number of individuals holding the minority opinion is denoted by $\hZ^{N, \textnormal{min}}_n$ and defined as in the non-delayed case.
Denote the number of opinion updates happening between times $k$ and $m$, $k<m$, by
\begin{equation*}
	\hS_{[k,m]}^{N, \textnormal{op}} \defeq \sum_{n=k}^m X_n^{N, \textnormal{op}}
\end{equation*}
and set $\hS_m^{N, \textnormal{op}} = \hS_{[1,m]}^{N, \textnormal{op}}$.
Note that for the opinions it holds that
\begin{equation*}
	\hZ^{N}_n = \hZ^N_0 + \sum_{i=1}^n X_i^{N, \textnormal{op}} \big(2\hV_i^N - 1\big).
\end{equation*}
The process $\hZ^N$ is a delayed simple symmetric random walk and the distribution of $\hZ^N_n$ equals the distribution of a simple symmetric random walk at time $\hS_n^{N, op}$.

\section{Proofs}\label{sec:proofs}
Before we start with the proof of Theorem \ref{thm:split} we want to note that the asymptotics of the time to absorption at $\{0, N\}$ of a simple symmetric random walk started in $\lfloor \frac{N}{2} \rfloor$ given in \eqref{eq:spitzer} were originally only formulated in \cite{Spitzer1964} for $N \to \infty$ with $N$ even.
For the proof of Theorem \ref{thm:split}, we will need a minor generalization:
\begin{lemma}\label{lem:spitzer}
	Let $(S_n)_{n \in \N_0}$ be a simple symmetric random walk started in $0$. For $\delta >0$, let 
	\begin{equation*}
		\tau_N^\delta \defeq \inf \Big\{n \geq 0: S_n \notin \left(-\delta\tfrac{N}{2},  \delta\tfrac{N}{2}\right)\Big\} .
	\end{equation*}
	Then
	\begin{equation}\label{lem:spitzer_fine}
		\lim_{N \to \infty} \IP \big( \tau_{N}^\delta > N^2 x\big) = \beta \left( \tfrac{x}{\delta^2}\right) .
	\end{equation}
\end{lemma}
\begin{proof}
	By a monotonicity argument one can drop the assumption on $N$ being even in the prelimit in \cite{Spitzer1964}.
	A substitution yields the quadratic dependence on $\delta$.
\end{proof}

This result demonstrates that the hitting times $\tau_N^\delta$ live on the $N^2$ scale in the sense that $\tau_N^\delta/N^2$ converges in distribution to a non-trivial limit.
The next lemma shows that the boundary hitting time is asymptotically smaller when the random walk is not centred initially.

\begin{lemma}\label{lem:non_centered_hitting_times}
	Let $z \in \Z$ and $(S_n)_{n \in \N_0}$ be a simple symmetric random walk started at $0$. For $\delta > 0$, let
	\begin{equation*}
		\tau_N^\delta(z) \defeq \inf \Big\{n \geq 0: S_n+z \notin \left(-\delta\tfrac{N}{2},  \delta\tfrac{N}{2}\right)\Big\}. 
	\end{equation*}
	Then,
	\begin{equation*}
		\limsup_{N \to \infty} \IP \big( \tau_N^\delta (z) > N^2 x \big) \leq \beta (\tfrac{x}{\delta^2}).
	\end{equation*}
\end{lemma}
\begin{proof}
	For $z\not\in \left(-\delta\tfrac{N}{2},  \delta\tfrac{N}{2}\right)$, the statement is trivial. If $z$ is even, then one can couple $S_n$ with another simple symmetric random walk $S'_n$ such that
	\begin{equation*}
		S'_n =-S_n \text{ for all }n \leq \inf\{k \geq 0 : S_n+z = S_n'\} 
	\end{equation*}
	and
	\begin{equation*}
		S'_n =S_n + z\text{ for all }n > \inf\{k \geq 0 : S_n+z = S_n'\}. 
	\end{equation*}
	This way, for all $n \leq \tau_N^\delta (z)$ it holds that
	\begin{equation*}
		|S_n+z - \delta \tfrac{N}{2} | \wedge |S_n+z + \delta \tfrac{N}{2} | \leq |S'_n - \delta \tfrac{N}{2} | \wedge |S'_n + \delta \tfrac{N}{2} |
	\end{equation*}
	and hence $\tau_N^\delta (z) \leq \tau_N^\delta$ which implies the statement for $z$ even.

	For $z$ odd, a first step analysis shows that
	\begin{align*}
		\tau_N^\delta (z) 
		&\overset d= 1 + \1_{\{S_1=-1\}}\tau_N^\delta (z-1) + \1_{\{S_1=1\}}\tau_N^\delta (z+1)
		\preceq  1 + \tau_N^\delta
	\end{align*}
	by the first part, where $\preceq$ denotes the stochastic order with respect to $\IP$.
	Hence,
	\begin{align*}
		\IP(\tau_N^\delta(z) > N^2x)
		&\leq \IP(\tau_N^\delta > N^2x - 1) 
		= \IP(\tau_N^\delta > N^2(x - \tfrac1{N^2})).
	\end{align*}
	Considering continuity and monotonicity of $\beta$ concludes the proof.
\end{proof}

In Section~\ref{sec:ProofConsensus} we consider a random walk running for a time in $\omega(N^2)$. Heuristically, this should hit the boundary with high probability. For the sake of completeness we show this fact in the following lemma.
\begin{lemma}\label{lem:AsymptoticsForDifferentScale}
	Let $z\in \Z$ and let $(S_n)_{n \in \N_0}$ be a simple symmetric random walk started at $0$. Further, let $(a_N)_{N\in \N}$ be a positive sequence such that $a_N\to \infty$ as $N\to \infty$. Then, for $\delta>0$, it follows that
	\begin{equation*}
		\lim_{N\to \infty} \IP\big(\tau_{N}^\delta(z) > a_NN^2\big)=0.
	\end{equation*}
\end{lemma}

\begin{proof}
	Again, for $z\not\in \left(-\delta\tfrac{N}{2},  \delta\tfrac{N}{2}\right)$ the statement is trivial. Hence, assume that $z\in \left(-\delta\tfrac{N}{2},  \delta\tfrac{N}{2}\right)$.
	Since $a_N\to \infty$ it follows by monotonicity and Lemma~\ref{lem:non_centered_hitting_times} that
	\begin{equation*}
		\limsup_{N\to\infty} \IP \big(\tau_{N}^\delta(z) > a_NN^2\big) 
		%\geq 
		%\rap{\leq} 
		\leq
		\limsup_{N\to \infty} \IP(\tau_{N}^\delta(z) > xN^2)
		\leq \beta \left( \tfrac{x}{\delta^2}\right)
	\end{equation*}
	for all $x>0$. By the definition of $\beta$ in equation \eqref{eq:spitzer} it is clear that $\beta(y)\to 0$ as $y\to\infty$.
\end{proof}
\subsection{Proofs for segregation and almost consensus}\label{sec:ProofsSegregation}
In this section, we prove Theorems~\ref{thm:split}, \ref{thm:split2} as well as Propositions~\ref{prop:almost_consensus} and \ref{prop:more_opinions}.

\begin{proof}[Proof of Theorem~\ref{thm:split}]
	In this proof we just need to consider some macroscopic quantities defined in Section~\ref{sec:construction} and do not need to track the complete underlying graph structure. To be precise, we consider
	\begin{itemize}
		\item $Z_n^N$ the number of individuals with opinion $1$ at time $n$,
		\item $S_n^{N, \textnormal{op}}$ the number of opinion updates until time $n$ and
		\item $S_n^{N, \textnormal{del}}$ the number of edge deletions until time $n$.
	\end{itemize}
	As already mentioned in Remark~\ref{rem:ExtensionBeyondAbsorption} we use the extensions of the processes $Z^N, S^{N, \textnormal{op}}$ and $S^{N, \textnormal{del}}$ beyond the absorption time and we will make use of the fact that $(Z_n^N)_{n \in \N_0}$ is equal in law to $(Y_{S_n^{N, \textnormal{op}}} + Z_0^N)_{n \in \N_0}$ where $Y$ is a simple symmetric random walk independent of $S_n^{N, \textnormal{op}}$ with initial state $Y_0 = 0$.

	A priori, the macroscopic quantities do not expose much information about the offended voter model.
	However, if we know that $Z^N$ has never left the interval $(c'N,(1-c')N )$
	up until time $n$ while $S_n^{N,\textnormal{del}}> \tbinom{N}{2}$, then the offended voter model must have terminated in the event of segregation before time $n$ since the edge set only contains $\tbinom{N}{2}$ edges.
	Moreover, it holds that
	\begin{equation}\label{eq:SuffSegregationCrit}
		\begin{aligned}
			\bigcup_{n=1}^\infty \Big\{ S_n^{N, \textnormal{del}}> \tbinom N2 &\text{ and } Z_k^N \in (c'N,(1-c')N )
			%\{\lceil c'N \rceil,\ldots, \lfloor (1-c')N \rfloor\}
			\text{ for } k\leq n \Big\}\\
			%k=1,\ldots,n \Big\}\\
			&\subset \Big\{ Z_\infty^{N, \textnormal{min}} > c'N \Big\} .
		\end{aligned}
	\end{equation}
	for any $c' < c\leq\tfrac{1}{2}$. In Lemma~\ref{lem:vot_del_moves} below we show, for any $r>1$, that if $\big\lceil r \frac{q}{1-q} \frac{N^2}{2}\big\rceil$ opinion updates took place, then with high probability at least $\tbinom{N}{2}$ deletions have happened. Thus, we use \eqref{eq:SuffSegregationCrit}, to see that
	\begin{alignat*}{2}
		& \liminf_{N \to \infty} \IP \big( Z_\infty^{N, \textnormal{min}} > c'N \big)\\ 
		\geq &\liminf_{N \to \infty} \IP \biggl( \bigcup_{n=1}^\infty \Big\{S_n^{N,\textnormal{del}}> \tbinom{N}{2}, \ S_n^{N, \textnormal{op}}= \Big\lceil r \tfrac{q}{1-q}\tfrac{N^2}{2} \Big\rceil, Z_k^N \in (c'N,(1-c')N ) \text{ for all } k\leq n \Big\} \biggr)\\
		\geq &\liminf_{N \to \infty} \IP \biggl( \bigcup_{n=1}^\infty \Big\{S_n^{N,\textnormal{op}}= \Big\lceil r \tfrac{q}{1-q}\tfrac{N^2}{2} \Big\rceil,  Z_k^N \in (c'N,(1-c')N ) \text{ for all } k\leq n \Big\} \biggr)\\
		&\qquad-\limsup_{N \to \infty} \IP \biggl( \bigcup_{n=1}^\infty \Big\{S_n^{N,\textnormal{del}} \leq \tbinom{N}{2}, S_n^{N,\textnormal{op}}= \Big\lceil r \tfrac{q}{1-q}\tfrac{N^2}{2} \Big\rceil \Big\} \biggr)
	\end{alignat*}
	for any $r > 1$. The limit superior equals $0$ which follows from Lemma~\ref{lem:vot_del_moves} and by the monotonicity of $S_n^{N,\textnormal{del}}$ and $S_n^{N,\textnormal{op}}$. Next we use that $Z_k^N$ has the same distribution as $Y_{S_k^{N, \textnormal{op}}} + Z_0^N$. Therefore, letting $T^N_r=\inf\{n:S_n^{N,\textnormal{op}}=\lceil r\frac{q}{1-q}\frac{N^2}{2}\rceil\}$ (cf.\ Lemma~\ref{lem:vot_del_moves}), the whole expression equals
	\begin{align*}
		&\liminf_{N \to \infty} \IP \left( \bigcup_{n=1}^\infty
		\{T^N_r=n\}\cap
		\Big\{
		%{\color{Felix}Z^N_k}
		Y_k + Z_0^N
		\in (c'N,(1-c')N )  \text{ for } k\leq \Big\lceil r\tfrac{q}{1-q}\tfrac{N^2}{2} \Big\rceil \Big\} \right)\\
		&= \liminf_{N \to \infty} \IP \left(Y_k + Z_0^N  \in (c'N,(1-c')N )  \text{ for } k\leq \Big\lceil r\tfrac{q}{1-q}\tfrac{N^2}{2} \Big\rceil \right).
	\end{align*}
	By symmetry we can assume without loss of generality that opinion $1$ is the minority opinion at time $0$. Then the last expression equals 
	\begin{align*}
		&\liminf_{N \to \infty} \IP \left(Y_k \in ( c'N  - Z_0^N ,  (1-c')N  - Z_0^N ) \text{ for } k\leq \Big\lceil r\tfrac{q}{1-q}\tfrac{N^2}{2} \Big\rceil \right)\\
		&\geq \liminf_{N \to \infty} \IP \left(Y_k \in \big( (c' - c)N , (c-c')N ) \text{ for } k\leq \Big\lceil r\tfrac{q}{1-q}\tfrac{N^2}{2}\Big\rceil \right)\\
		&\geq \beta \left( \tfrac{r}{2} \tfrac{q}{1-q} \tfrac{1}{4 (c-c')^2}\right).
	\end{align*}
	For the first inequality, we consider $Z_0^N-c'N$ as the right border of the interval, which is strictly smaller than $(1-c')N-Z_0^N$ since $c'<c$. Furthermore, the probability is then minimized by choosing $Z_0^N=\lceil c N\rceil$. For the last inequality recall the definition of $\beta$ in \eqref{eq:spitzer} and Lemma~\ref{lem:spitzer}.
	
	Letting $r \searrow 1$, the desired result follows from monotonicity and continuity of $\beta$. Note that the monotonicity of $\beta$ follows from that of the prelimit in Lemma \ref{lem:spitzer}, $\IP \left(\tau_N^\delta > N^2 x\right)$.
\end{proof}

\begin{lemma}\label{lem:vot_del_moves}
	Let $r>1$, $q_N=q\in (0,1)$ or $q_N\to 0$ as $N\to \infty$. Further, let $\alpha_N=\max\{q_N,\frac1N\}$ and denote by $T_{r}^N$ the first time that the number of opinion updates hits $\big\lceil r \tfrac{\alpha_N}{1-q_N} \tfrac{N^2}{2} \big\rceil$, i.e.\
	\begin{equation}
		T_{r}^N = \inf \left\{n \in \N: S_n^{N, \textnormal{op}} =  \Big\lceil r \tfrac{\alpha_N}{1-q_N} \tfrac{N^2}{2} \Big\rceil \right\}.
	\end{equation}
	Then, w.h.p.\ $T_{r}^N \geq \big\lceil r \frac{\alpha_N}{1-q_N} \frac{N^2}{2}\big\rceil + \frac{N^2}{2}$ as $N \to \infty$
	and thus, $S^{N,\textnormal{del}}_{T^N_r}\geq \frac{N^2}2$ w.h.p., i.e.\ until time $T^N_r$ at least $N^2/2$ edge deletions occurred with high probability.
\end{lemma}
\begin{proof}
	Fix $r>1$ and set $X^N:=T_{r}^N - k_N =S^{N,\textnormal{del}}_{T^N_r}$, where $k_N=\lceil r \tfrac{\alpha_N}{1-q_N} \tfrac{N^2}{2} \rceil$. We can interpret $X^N$ as a negative binomial random variable. This can be seen by considering a single opinion update as a ``success'' and a deletion as ``failure''.
	Then $X^N$ counts how many ``failures'' occurred before the first $k_N$ ``successes'' took place.
	Since negative binomials arise as the sum of i.i.d.\ geometric random variables, for large parameters and a proper rescaling they are well-approximated by a normal distribution, which follows by an application of the Central Limit Theorem.
	In this case %Thus, for negative binomial random variables 
	$X^N \sim \text{nB}\big(k_N, q_N\big)$ has mean $\mu_N = \frac{1-q_N}{q_N}k_N$ and standard deviation $\sigma_N = \frac{\sqrt{\mu_N}}{\sqrt{q_N}}$.
	
	We now verify the Lyapunov condition in order to apply the Central Limit Theorem (e.g.\ \cite[Lemma~15.41, Theorem~15.43]{kenke2020probability}):
	Denote by $Y^N_1,\ldots,Y^N_{k_N}$ i.i.d.\ and $\text{Geo}(q_N)$-distributed random variables.
	Further, let $\widetilde Y^N_i$ the corresponding standardized variables, i.e.\ 
	\begin{equation*}
		\widetilde Y^N_i
		= \frac{Y^N_i - \IE[Y^N_i]}{\sqrt{\mathds{V}[X^N]}}
		= k_N^{-\frac12}\cdot \frac{Y^N_i - \IE[Y^N_i]}{\sqrt{\mathds{V}[\smash{Y^N_i}]}}.
	\end{equation*}
	Then, recalling the kurtosis of the $\text{Geo}(p)$-distribution is given by $9+\frac{p^2}{\sqrt{1-p}}$, we obtain the Lyapunov condition for $\delta=2$ since
	\begin{equation*}
		\sum_{i=1}^{k_N}\IE\big[|\widetilde Y^N_i|^4\big]
		= k_N^{-1}\cdot(9+\tfrac{q_N^2}{\sqrt{1-q_N}})
		\xrightarrow[]{N\to\infty} 0.
	\end{equation*}
	Now, observe that
	\begin{equation*}
		\IP\big( X^N \geq \tfrac{N^2}{2} \big)= \IP \Big( \tfrac{X^N-\mu_N}{\sigma_N} \geq \tfrac{\frac{N^2}{2}-\mu_N}{\sigma_N}\Big),
	\end{equation*}
	where the right hand side in the last event equals
	\begin{equation*}
		k_N^{-\frac{1}{2}}\bigg(\frac{q_N}{\sqrt{1-q_N}}\frac{N^2}{2}-\frac{1-q_N}{\sqrt{1-q_N}}k_N\bigg)
		\sim N\cdot\frac{q_N-r\alpha_N}{\sqrt{2r\alpha_N}}
		\leq N\cdot (1-r)\sqrt{\alpha_N}
		\xrightarrow{N\to\infty} -\infty,
	\end{equation*}
	since $r>1$. Thus, the desired probability converges to $1$ as $N \to \infty$.
\end{proof}
\begin{proof}[Proof of Theorem~\ref{thm:split2}]
	Let us recall that we chose $c>c'$, $Z_0^{N, \textnormal{min}}>cN$ and $q_N\to 0$ as $N\to \infty$. Now it follows in the exact same way as in the proof of Theorem~\ref{thm:split} that
	\begin{align*}
		&\liminf_{N \to \infty} \IP \big( Z_\infty^{N, \textnormal{min}} > c'N \big)\\
		&\geq \liminf_{N \to \infty} \IP \left(Y_i+ Z_0^N \in ( c'N ,  (1-c')N ) \text{ for } i=1,\ldots,\Big\lceil r\tfrac{\alpha_N}{1-q_N}\tfrac{N^2}{2} \Big\rceil \right),
	\end{align*}
	where $(Y_i)_{i\geq 0}$ is a simple symmetric random walk starting at $0$. Note that, since $q_N,\alpha_N\to 0$, for every $r^*>0$ there exists an $N_{*}$ such that $\tfrac{r}{2} \tfrac{\alpha_N}{1-q_N}\leq r^*$ for all $N\geq N^*$. Thus, it follows for all $r^*>0$ that
	\begin{align*}
		&\liminf_{N \to \infty} \IP \big( Z_\infty^{N, \textnormal{min}} > c'N \big)\\
		&\geq \liminf_{N \to \infty} \IP \left(Y_i+ Z_0^N \in ( c'N ,  (1-c')N ) \text{ for } i=1,\ldots,\big\lceil r^*N^2 \big\rceil \right)\\
		&\geq\beta \Big( \frac{r^*}{4 (c-c')^2}\Big),
	\end{align*}
	where we used monotonicity of the probability measure $\IP$ in the first inequality. Analogously as in Theorem~\ref{thm:split}, the second inequality follows by an application of Lemma~\ref{lem:spitzer}. Finally, by definition of $\beta$ (see \eqref{eq:spitzer}) it follows that $\beta(x)\to 1$ as $x\to 0$. Thus, by letting $r^*\to 0$ we obtain the claim.
\end{proof}
\begin{proof}[Proof of Proposition~\ref{prop:almost_consensus}]
	The dynamics can only terminate in a state with final minority opinion $Z_\infty^{N, \textnormal{min}}> \varepsilon N$, if at least $\varepsilon (1-\varepsilon) N^2$ edges have been deleted. This is because at absorption there cannot be any edges between opposing opinions. That means that none of the $Z_\infty^{N, \textnormal{min}}$ vertices with minority opinion can be connected to any of the $N-Z_\infty^{N, \textnormal{min}}$ vertices with opposing opinion, so
	\begin{equation*}
		Z_\infty^{N, \textnormal{min}} \bigl(N-Z_\infty^{N, \textnormal{min}} \bigr) 
		\geq \varepsilon N (N- \varepsilon N)
	\end{equation*}
	edges must have been deleted.
	
	But then, by exchanging $q$ for $1-q$ and the roles of $S^{N,\textnormal{op}}$ and $S^{N,\textnormal{del}}$ in Lemma~\ref{lem:vot_del_moves}, with high probability $\lceil r \frac{q}{1-q} N^2 \varepsilon (1-\varepsilon) \rceil$ opinion updates must have happened for any $r \in (0,1)$.
	Lemma~\ref{lem:non_centered_hitting_times} shows that the probability that the random walk of the opinions hits $\{0,N\}$ earlier than that is asymptotically bounded from below by
	\begin{equation*}
		1 - \beta \left(r\frac{q}{1-q}\varepsilon (1-\varepsilon)\right) .
	\end{equation*}
	Letting $r \to 1$, one obtains the desired result.
\end{proof}
We finish this subsection with a sketch of the proof that segregation occurs in an offended voter model with multiple opinions with positive probability for $q$ small enough.
\begin{proof}[Proof sketch of Proposition~\ref{prop:more_opinions}]
	Note that
	\begin{align*}
		\IP \left(\{Z_\infty^{N, \textnormal{min}} > \delta N\}\right) =
		\IP &\Bigg( \bigcap_{k \in [K]} \{Z_\infty^{N,k} \in (\delta N, (1-\delta)N) \}\Bigg)\\
		\geq 1&- \sum_{k \in [K]} \IP \left( Z_\infty^{N,k} \notin (\delta N, (1-\delta)N) \right).
	\end{align*}
	Fix $r > 1$.
	By a similar argument as in Subsection~\ref{sec:ProofsSegregation}, the delayed random walk $Z_n^{N,k}$ w.h.p.\ gets to jump no more than $\big\lceil r \frac{q}{1-q} \frac{N^2}{2}\big\rceil$ times before absorption. To see this, note that no more than $\frac{N^2}{2}$ deletions can occur before absorption.
	Since the ratio of deletions and opinion updates concentrates, at most $\big\lceil r \frac{q}{1-q} \frac{N^2}{2}\big\rceil$ opinion updates happen during this time with high probability.
	For any $k$, the delayed random walk $Z_n^{N,k}$ can only jump when there is an opinion update.
	Hence, for any $k$, the total numbers of jumps in $Z_n^{N,k}$ is w.h.p.\ bounded from above by $\big\lceil r \frac{q}{1-q} \frac{N^2}{2}\big\rceil$.
	Therefore,
	\begin{align*}
		\limsup_{N \to \infty} \IP &\left( Z_\infty^{N,k} \notin (\delta N, (1-\delta)N )\right)\\
		\leq \lim_{N \to \infty} \IP &\left( Z_0^{N,k} + X_n \notin (\delta N, (1-\delta)N) \text{ for some } n \leq \big\lceil r \tfrac{q}{1-q} \tfrac{N^2}{2}\big\rceil\right)
	\end{align*}
	For $q>0$ small enough, this limit is smaller than $K^{-1}$ by Lemma~\ref{lem:spitzer}.
	Thus, there exists a $q_K^*>0$ such that for all $q \leq q_K^*$
	\begin{equation*}
		\liminf_{N \to \infty} \IP \left(\{Z_\infty^{N, \textnormal{min}} > \delta N\}\right) > 1 - K\cdot K^{-1} = 0. \qedhere
	\end{equation*}
\end{proof}

\subsection{Proof of connected consensus}\label{sec:ProofConsensus}

In this section, we prove Theorem~\ref{thm:consensus}. We start by giving an outline of the proof.

The proof uses the dynamical deletion graph defined in Section~\ref{sec:construction} to prove Theorem~\ref{thm:consensus} for the delayed OV-model, which implies the theorem to hold for the OV-model.
Recall that at all times the dynamical deletion graph is a subgraph of the delayed OV-Model.
The general idea for showing (dense) connectedness of the final graph is to study the connectedness of the dynamical deletion graph and argue that the delayed random walk performed by the opinions hits $0$ or $N$ while the graph is still (densely) connected.

Bounds for the time until which the dynamical deletion graph is still densely connected can be obtained via comparison with an \ER graph, see Lemmas~\ref{lem:er_degrees} and~\ref{lem:ConnectivityAndMinimalDegree}.
Furthermore, absorption times of a non-delayed random walk are - as can be seen in Lemma~\ref{lem:spitzer} - well studied.
The main difficulty lies in controlling the delay of the random walk.

The delay depends directly on the number of discordant edges left in the graph:
If this number decreases, the probability that a randomly sampled edge is a discordant edge decreases.
An opinion update is only performed if the sampled edge is a discordant edge.
Thus, less opinion updates are performed if the number of discordant edges gets small.
This is equivalent to an increased time delay in the opinion random walk.
The proportion of discordant edges, however, depends heavily on the underlying graph structure of the model and is therefore difficult to control.

In the setting $1-q_N \in O(N^{-\delta})$ of Theorem~\ref{thm:consensus} this can be achieved in the following way:
First, subdivide the domain of the random walk $\{0,\dots,N\}$ into finitely many subdomains or ``levels'', as visualised in Figure~\ref{fig:levels}. These levels were chosen in such a way that if level $k$ is hit then with high probability level $k+1$ is hit instead of returning to level $k-1$.
Next note that for the number of discordant edges it holds that
\begin{equation*}
	\hE_n^{N,d} \geq \hZ_n^{N, \textnormal{min}} \Bigl( \hD_n^{N, \textnormal{min}} - \hZ_n^{N, \textnormal{min}}\Bigr).
\end{equation*}
This is because each of the $\hZ_n^{N, \textnormal{min}}$ vertices that hold the minority opinion has at least $\hD_n^{N, \textnormal{min}}$ neighbours out of which at least $\bigl( \hD_n^{N, \textnormal{min}} - \hZ_n^{N, \textnormal{min}}\bigr)$ have the other opinion.
%For each level
By definition of the levels we have control over $\hZ_n^{N, \textnormal{min}}$, and $\hD_n^{N, \textnormal{min}}$ can be bounded by using a comparison result with an \ER graph.
Therefore, for each level the number of discordant edges $\hE_n^{N,d}$ and with it the time delay of the random walk can be bounded from below.
Knowing the lower bound of the delay, one can bound the time it takes to reach level $k+1$ from level $k$ for all $k$.
Since only finitely many levels have to be crossed until absorption, the total time to absorption can then be asymptotically bounded from above.

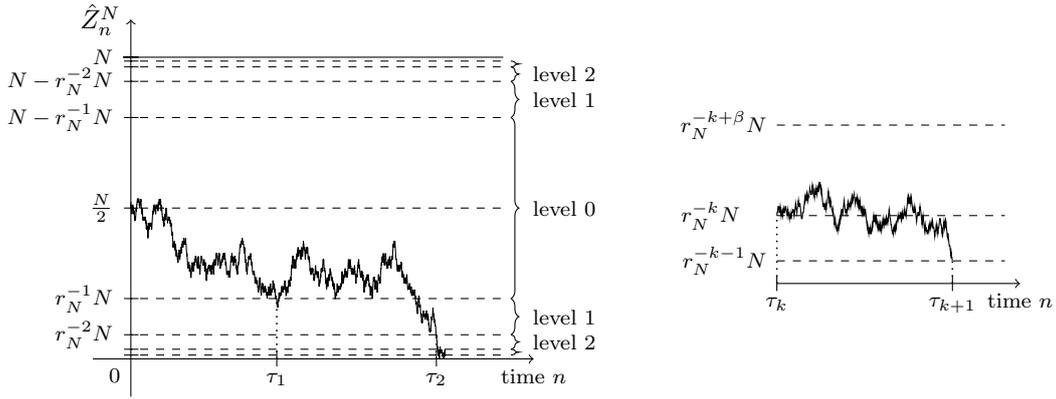
\begin{figure}[htb]
    \centering
    \begin{tikzpicture}[every node/.style={font=\scriptsize}]
    % Draw x-axis and y-axis with arrows
    \draw[->] (-0.5,0) -- (5.3,0) node[anchor=north] {\scriptsize time $n$};
    \draw[->] (0,-0.5) -- (0,4.5) node[anchor=east] {\small{$\Hat{Z}_n^N$}};

    % Draw ticks on x-axis
    %\foreach \x in {1, 2, 3, 4}
        
        %\draw (\x,0.1) -- (\x,-0.1) node[anchor=north] {\x};

    \draw (0.1,4) -- (-0.1,4) node[anchor=east] {$N$} -- ++(5,0);
    \draw (0.1,2) -- (-0.1,2) node[anchor=east] {$\frac{N}{2}$} [dashed] -- ++(5,0);;
    \draw (0.1,0.8) -- (-0.1,0.8) node[anchor=east] {$r_N^{-1} N$} [dashed] -- ++(5,0);;
    \draw (0.1,0.32) -- (-0.1,0.32) node[anchor=east] {$r_N^{-2} N$} [dashed] -- ++(5,0);;
    \draw (0.1,0.128) -- (-0.1,0.128) node[anchor=east] {} [dashed] -- ++(5,0);;
    \draw (0.1,0.0512) -- (-0.1,0.0512) node[anchor=east] {} [dashed] -- ++(5,0);;

    \draw (0.1,4-0.8) -- (-0.1,4-0.8) node[anchor=east] {$N-r_N^{-1} N$} [dashed] -- ++(5,0);;
    \draw (0.1,4-0.32) -- (-0.1,4-0.32) node[anchor=east] {$N-r_N^{-2} N$} [dashed] -- ++(5,0);;
    \draw (0.1,4-0.128) -- (-0.1,4-0.128) node[anchor=east] {} [dashed] -- ++(5,0);;
    \draw (0.1,4-0.0512) -- (-0.1,4-0.0512) node[anchor=east] {} [dashed] -- ++(5,0);;

    % First version
    % \draw[dotted, line width=0.6pt] (3.85,0) -- ++(0,0.8);
    % \draw (3.85,-0.1) -- ++(0,0.2) node[below, yshift=-0.15cm] {$\tau_1$\phantom.};

    % \draw[dotted, line width=0.6pt] (4.08,0) -- ++(0,0.32);
    % \draw (4.08,-0.1) --++ (0,0.2) node[below, yshift=-0.15cm] {\phantom.$\tau_2$};

    % Felix tries a shift
    \draw[dotted, line width=0.6pt] (1.92,0) -- ++(0,0.8);
    \draw (1.92,-0.1) -- ++(0,0.2) node[below, yshift=-0.15cm] {$\tau_1$};

    \draw[dotted, line width=0.6pt] (4.02,0) -- ++(0,0.32);
    \draw (4.02,-0.1) --++ (0,0.2) node[below, yshift=-0.15cm] {$\tau_2$};

    % Label the origin
    \node[anchor=north east] at (0,0) {$0$};

    \draw [decorate, decoration={brace, amplitude=3pt}] (5,4-0.8) -- (5,0.8) 
        node [midway, xshift=20pt] {level $0$};
    \draw [decorate, decoration={brace, amplitude=3pt, mirror}] (5,4-0.8) -- (5,4-0.32) 
        node [midway, xshift=20pt] {level $1$};
    \draw [decorate, decoration={brace, amplitude=3pt, mirror}] (5,4-0.32) -- (5,4-0.128) 
        node [midway, xshift=20pt] {level $2$};

    \draw [decorate, decoration={brace, amplitude=3pt, mirror}] (5,4-0.128) -- (5,4-0.0512) 
        node [midway, xshift=20pt] {};

     \draw [decorate, decoration={brace, amplitude=3pt}] (5,0.8) -- (5,0.32) 
         node [midway, xshift=20pt] {level $1$};
     \draw [decorate, decoration={brace, amplitude=3pt}] (5,0.32) -- (5,0.128) 
         node [midway, xshift=20pt] {level $2$};
     \draw [decorate, decoration={brace, amplitude=3pt}] (5,0.128) -- (5,0.0512) 
         node [midway, xshift=20pt] {};

    % first version:
    % \draw[black] plot file {plots/raw_data/plot_rw1};

    % Felix tries a shift    
    \draw[black, xshift=-0.0432cm, yshift=-0.2cm] plot file {plots/raw_data/plot_rw1b};

    \begin{scope}[xshift=6.5cm]
    \draw[dashed] (2,3.1) node[anchor=east] {\scriptsize{$r_N^{-k+\beta}N$}} -- ++(3,0) ;
    \draw[dashed] (2,1.9) node[anchor=east] {\scriptsize{$r_N^{-k}N^{\phantom{-1}}$}} -- ++(3,0);
    \draw[dashed] (2,1.3) node[anchor=east] {\scriptsize{$r_N^{-k-1}N$}} -- ++(3,0);
    \draw[black, xshift=2cm] plot file {plots/raw_data/plot_rw2};
    \draw[->] (2,1) -- ++(3.2,0) node[anchor=north] {\scriptsize time $n$};
    
    \draw[dotted, line width=0.6pt] (4.31,1) -- ++(0,0.3);
    \draw (4.31,0.9) --++ (0,0.2) node[below, yshift=-0.15cm] {$\tau_{k+1}$};
    
    \draw[dotted, line width=0.6pt] (2,1) -- ++(0,0.9);
    \draw (2,0.9) -- ++(0,0.2) node[below, yshift=-0.15cm] {$\tau_k$};

    \end{scope}

\end{tikzpicture}
    \caption{Left: Subdivision of the domain of the random walk into finitely many ``levels''. Right: Starting from level $k$, the level $k+1$ is hit w.h.p.\ before returning to $r_N^{-k+\eta}N$, $\eta\in(0,1)$, (or $N- r_N^{-k+\eta}N$) and hence before returning to level $k-1$.
    }\label{fig:levels}
\end{figure}

\medskip

We say that the random walk is on the $k$-th level if $r_N^{-(k+1)} N\leq \hZ_n^{N,\textnormal{min}} \leq r_N^{-k} N$ and we denote the first hitting time of the $k$-th level by
\begin{equation*}
	\tau_k \defeq \min \{n \geq 0 : \hZ_n^{N,\textnormal{min}} \leq r_N^{-k} N\},
\end{equation*}
where $r_N = (1-q_N)^{-\frac{1}{2}}$. Note that for our specific choice of $q_N$ and $r_N$, the levels are chosen such that there exists a $k_0  \in \N$  such that $r_N^{-k} N \to 0$ for all $k > k_0$ and $\liminf_{N \to \infty} r_N^{-k} N > 0$ for all $k \leq k_0$.
Also note that level $0$ has linear size in $N$ while all higher levels are sublinear.

First we need some control over the minimal degree of a dense \ER graph. This is provided by the following result in the form of an asymptotic lower bound. %on the minimal degree.
\begin{lemma}[\ER Degrees]\label{lem:er_degrees}
	Let $ \bigl( \Tilde{G}^N \bigr)_{N \in \N_0}= \bigl(\Lambda^N, \Tilde{\cE}^N\bigr)_{N \in \N_0}$ be a sequence of \ER graphs with edge probabilities $p \in (0,1)$. Denote by
	\begin{equation*}
		\cN_x^N = \big\{ y \in \Lambda^N : \{x, y\} \in \Tilde{\cE}^N\big\}
	\end{equation*}
	the neighborhood of vertex $x \in \Lambda^N$ in $\Tilde{G}^N$. Then, for any $\varepsilon > 0$, it holds that
	\begin{equation*}
		\IP\Big(\min_{x \in \Lambda^N} \big\lvert \cN_x^N \big\lvert \geq (p-\varepsilon)N \Big)\to 1 \quad \text{ as } \quad N\to \infty.
	\end{equation*}
\end{lemma}

\begin{proof}
	Note that for any vertex $x \in \Lambda^N$, its number of neighbours follows a Binomial distribution, that is $|\cN_x^N| \sim \text{Bin} (N-1, p)$.
	Furthermore, by the large deviations result given in Theorem 2.19 of \cite{van2016random}, it holds that
	\begin{equation}\label{eq:intersect}
		\IP \big(|\cN^N_x|\geq (p-\varepsilon)N\big)\geq 1-e^{-(N -1) I(\varepsilon)} - \rho_N
	\end{equation}
	for any $x \in \Lambda^N$, where $I(\cdot)$ is the rate function of a Binomial random variable and
	\begin{equation*}
		\rho_N \defeq \IP \big((p-\varepsilon)(N-1) \leq |\cN^N_x| \leq (p-\varepsilon)N \big).
	\end{equation*}
	Now we would like to consider the event that each vertex has more than $(p-\varepsilon)N$ neighbours, i.e. 
	\begin{equation*}
		\bigcap_{x\in \Lambda^N} \Big\{|\cN_x^N|\geq (p-\varepsilon)N \Big\}.
	\end{equation*}
	The events $\{|\cN_x^N|\geq (p-\varepsilon)N\}, x \in \Lambda^N$ are not independent so their probability is not given by the product of Binomial probabilities.
	
	However, in the vocabulary of percolation (see \cite{Grimmett1999}, p.\ 32), the events $\{|\cN_x^N|\geq (p-\varepsilon)N\}$ are increasing. To see this, note that $\IP^N$, i.e.\ the distribution of the \ER graph on $N$ vertices, can be written as the product measure $\IP^N = \bigotimes_{e \in E^N} \pi_p$ where $\pi_p \sim \text{Ber}(p)$. Then, $E'' \subset E' \subset E^N$ implies $|\cN^N_x (E'')| \leq |\cN^N_x (E')|$ where $|\cN^N_x (E')|$ is the number of neighbours of $x$ in the configuration $E'$. For increasing events the FKG-inequality holds (see \cite{Grimmett1999}, p.\ 34), yielding
	\begin{equation*}
		\IP^N\bigg(\bigcap_{x\in \Lambda^N}\{|\cN^N_x|\geq (p-\varepsilon)N\}\bigg) \geq \prod_{x \in \Lambda^N} \IP^N(|\cN^N_x|\geq (p-\varepsilon)N) \geq \left( 1-e^{- (N-1) I(\varepsilon)} - \rho_N\right)^N.
	\end{equation*}
	By Theorem 2.21 in \cite{van2016random}, the rate function is given by $I(\varepsilon) = \frac{\varepsilon^2}{2p}$.
	The De-Moivre-Laplace limit theorem implies that $\rho_N \leq c_1 N^{-\frac{1}{2}}\exp{(-c_2 N)}$ for appropriate $c_1, c_2 > 0$.
	Together with the Bernoulli inequality, we get that the above term is greater than or equal to
	$1-N\exp(-\tfrac{(N-1)\varepsilon^2}{2p}) - N \rho_N$, which tends to $1$ as $N\to\infty$.
\end{proof}
In the next result we couple the dynamical deletion graph with an \ER graph. This allows us to obtain a sufficient criterion on the number of time steps $n$ to guarantee connectedness of the induced graph. In combination with Lemma~\ref{lem:er_degrees} we can obtain lower bounds on the minimal degree.
\begin{lemma}[Connectivity of $\cE^N$]\label{lem:ConnectivityAndMinimalDegree}
	Let $\cE^N_0 = E^N$ and $q_N < 1$ for all $N \in \N$. Let $\varepsilon>0$. The dynamical deletion graphs $(\Lambda^N,\cE_{n}^N )_{n\in\N_0}$ are connected with high probability for all 
	\begin{equation*}
		n\leq n_1(N,\varepsilon) \defeq \binom{N}{2}\frac{\log(N)-\log((1+\varepsilon)\log(N))}{2(1-q_N)}.
	\end{equation*}
	That is,
	\begin{equation*}
		\lim_{N \to \infty} \IP \Bigl( \bigl(\Lambda^N,\cE^N_n \bigr) \text{ is connected for all } n \leq n_1 (N, \varepsilon) \Bigr) = 1
	\end{equation*}
	Furthermore, for $\kappa\in (0,1)$ let $\varepsilon=\varepsilon(\kappa)>0$ be chosen such that $\kappa+\varepsilon<1$ and set
	\begin{equation*}
		n_2(N,\kappa,\varepsilon)\defeq \binom N2\frac{\log(\frac1{\kappa+\varepsilon})}{2(1-q_N)}.
	\end{equation*}
	Then it follows that
	\begin{align*}
		\lim_{N \to \infty} \IP\Big(\min_{x\in \Lambda^N}|\{y\in \Lambda^N: \{x,y\}\in \cE_{n}\}|\geq \kappa N \text{ for all } n \leq n_2 (N, \kappa, \varepsilon)\Big) = 1.
	\end{align*}
	In particular, the graph $(\Lambda^N,\cE^N_n)$ is connected and the minimal degree is bounded from below by $\kappa N$ for all $n\leq n_2(N,\kappa,\varepsilon)$ with high probability as $N\to \infty$.
\end{lemma}

\begin{proof}
	Let $(T_n^N)_{n \in \N_0}$ be i.i.d.\ $\text{Exp}(|E^N|)$ random variables and define $\Tilde{\cE}_t^N \defeq \cE^N_n$ for $t\in [t_n^N, t_{n+1}^N)$ where $t_k^N = \sum_{i=1}^{k}T_i^N$ and $t_0^N=0$. Now $\Tilde{\cE}_t^N$ is an \ER graph for all $t \geq 0$ with edge probability $p^N_t=\exp\big(-(1-q_N)t\big)$.
	To see this, note that instead of having a single $\text{Exp}(|E^N|)$ clock and selecting a random pair of vertices whenever the clock rings, one can equip each pair of vertices with an $\text{Exp}(1)$ clock.
	Whenever a clock rings, the corresponding edge is deleted independently with probability $1-q_N$.
	So, effectively, every edge has an independent rate $1-q_N$ deletion clock.
	By the properties of the exponential distribution, the probability that a given edge is still present at time $t$ is $\exp\big(-(1-q_N)t\big)$, independently of all other edges.
	
	Note that for all
	\begin{equation*}
		t\leq t_1=t_1^{N,\varepsilon}:=\frac{\log N-\log((1+\varepsilon)\log(N))}{1-q_N}
	\end{equation*}
	it holds that $p^N_t\geq \frac{(1+\varepsilon)\log(N)}{N}$ and for all
	\begin{equation*}
		t\leq t_2=t_2^{N, \varepsilon,\kappa}:=\frac{\log(\frac1{\kappa+\varepsilon})}{1-q_N},
	\end{equation*}
	we get that $p_t^N \geq \kappa+\varepsilon$.
	
	It is well known that a sequence of \ER graphs with $N \to \infty$ vertices is connected with high probability if the edge probability is greater than or equal to
	$\frac{(1+\varepsilon)\log(N)}{N}$, see \cite[Theorem~5.8]{van2016random}.
	Therefore, $\Tilde{\cE}_t^N$ is connected w.h.p.\ until time $t_1$.
	By Lemma \ref{lem:er_degrees} we get that $\Tilde{\cE}_t^N$ w.h.p.\ has minimal degree greater $\kappa N$ until time $t_2$.
	
	It remains to translate the conditions ``$t\leq t_1$'' and ``$t \leq t_2$'' back to discrete time.
	Specifically, we show that $t^N_{n_1} \leq t_1$ and $t^N_{n_2} \leq t_2$ hold w.h.p.:
	Denote by $(P_t)_{t \geq 0}$ a Poisson process on $\IR^{+}$ with rate $|E^N|$. Then,
	\begin{equation*}
		\IP \Bigl(t_{n_1}^N \leq t_1 \Bigr)
		= \IP \left( \sum_{i=1}^{n_1} T_i^N \leq t_1\right)
		= \IP(P_{t_1} \geq n_1).
	\end{equation*}
	Next we see that $n_1=n_1(N,\varepsilon)=\frac{1}{2}|E_N|t_1$, and thus it follows that
	\begin{align*}
		\IP(P_{t_1} < n_1)
		&= \IP(P_{t_1} < \tfrac12|E^N|t_1)
		\leq e^{-\frac{1}{8} {|E^N|t_1}}
		\to 0,
	\end{align*}
	where we used the Chernoff bound for Poisson random variables given in \cite{Alon2008}, Theorem~A.1.15. An analogous calculation holds for $t_2$. Lastly, note that $n\leq n_2(N,\kappa,\varepsilon)$ implies $n\leq n_1(N,\varepsilon)$ for $N$ large enough since $n_1$ is of higher order.
\end{proof}

Lemma~\ref{lem:ConnectivityAndMinimalDegree} provides us with some control over connectedness and size of the minimal degree of the dynamical deletion graph.
The next step is to get control over the delay of the random walk. This is done by the following lemma.

\begin{lemma}\label{lem:iidcoupling}
	Let $f: \N \mapsto \IR^+$ be such that $f(N) \leq \frac{N(N-1)}{2}$. Then, for any $N \in \N$ there exists a sequence of i.i.d.\ Bernoulli random variables $(Y_n^N)_{n \in \N} = (Y^{N, f}_n)_{n \in \N}$ such that $\IP \bigl( Y^N_n = 1\bigr) = \frac{2 q_N f(N)}{N(N-1)}$ and
	\begin{equation*}
		X_n^{N, \textnormal{op}} \geq Y_n^N \quad\text{if}\quad \widehat{E}_{n-1}^{N,d} \geq f(N) .
	\end{equation*}
\end{lemma}
\begin{proof}
	Recall that $\hE_n^{N,d} = \lvert \bfhE_n^{N,d} \rvert$ denotes the number of discordant edges in the delayed model. Let $(W'_n)_{n\in \N}$ and $(W''_n)_{n\in \N}$ be two independent sequences of uniformly distributed random variables with values on $[0,1]$ such that they are independent of the sequences $(\hU^N_n)_{n\in\N}$, $(\hV^N_n)_{n\in\N}$ and the uniformly sampled edges for all time steps $n\in \N$. In other words, these two sequences are independent of all randomness used to construct the delayed OV-Model.
	Now we set
	\begin{equation*}
		Y'_n:=\1_{\big\{W_n'\leq \frac{f(N)}{ \hE_{n-1}^{N,d}}\big\}}\quad \text{ and } \quad Y''_n:=\1_{\Big\{W_n''\leq \frac{2(f(N)-\hE_{n-1}^{N,d})}{N(N-1)-2 \hE_{n-1}^{N,d}}\Big\}}
	\end{equation*}
	for all $n\in\N$. As a consequence of this definition $(Y'_n)_{n\in \N}$ and $(Y''_n)_{n\in \N}$ are two sequences of random variables with values in $\{0,1\}$ such that, for any $n\in\N$,
	$Y'_n$ and $Y''_n$ are conditionally on $\bfhZ_{n-1}^{N,d}$ independent of $(\bfhZ_k^{N,d})_{k\leq n-2}$, $X_n^{N, \textnormal{op}}$, $X_n^{N,\textnormal{e}}$ and of each other. Their conditional distributions are given by
	\begin{equation*}
		\IP(Y'_n=1 | \bfhZ_{n-1}^{N})=\frac{f(N)}{ \hE_{n-1}^{N,d}}\wedge 1
		\quad\text{and}\quad
		\IP(Y''_n=1 | \bfhZ_{n-1}^{N})=\frac{2(f(N)-\hE_{n-1}^{N,d})}{N(N-1)-2 \hE_{n-1}^{N,d}}\vee 0.
	\end{equation*}
	Now define the sequence
	\begin{align*}
		Y_n:=X_n^{N, \textnormal{op}}Y'_n\1_{\{ \hE_{n-1}^{N,d}\geq f(N)\}}+ \hU_n^N(1-(1-X_n^{N,\textnormal{e}})(1-Y''_n))\1_{\{ \hE_{n-1}^{N,d}< f(N)\}}.
	\end{align*}
	By definition and Remark~\ref{rem:delayed_ov_model} it follows that
	\begin{align*}
		\IP(&Y_n=1|\bfhZ_{n-1}^{N})\\
		&=  \IP(X_n^{N, \textnormal{op}}=1, Y'_n=1|\bfhZ_{n-1}^{N})\1_{\{ \hE_{n-1}^{N,d}\geq f(N)\}}\\
		&\quad +   \IP(\hU_n^N=1)(1-\IP(X_n^{N, \textnormal{op}}=0,Y''_n=0|\bfhZ_{n-1}^{N})\1_{\{ \hE_{n-1}^{N,d}< f(N)\}}\\
		&= \frac{2 q_N \hE_{n-1}^{N,d}}{N(N-1)}\frac{f(N)}{ \hE_{n-1}^{N,d}}\1_{\{ \hE_{n-1}^{N,d}\geq f(N)\}}\\
		&\quad +q_N\bigg(1-\Big(1-\frac{2 \hE_{n-1}^{N,d}}{N(N-1)}\Big)\frac{N(N-1)-2f(N)}{N(N-1)-2 \hE_{n-1}^{N,d}}\bigg) \1_{\{\hE_{n-1}^{N,d}< f(N)\}}\\
		&= \frac{2q_Nf(N)}{N(N-1)}\1_{\{ \hE_{n-1}^{N,d}\geq f(N)\}}+\frac{2q_Nf(N)}{N(N-1)}\1_{\{ \hE_{n-1}^{N,d}< f(N)\}}
		=\frac{2q_Nf(N)}{N(N-1)}.
	\end{align*}
	Since the right hand side does no longer depend on $\bfhZ_{n-1}^{N}$ it follows that $(Y_n)_{n\geq 0}$ is a sequences of Bernoulli variables with success probability $\frac{2q_Nf(N)}{N(N-1)}$. 
	We only show pairwise independence, that is for $m \neq n$. Analogously it follows that arbitrary finite subsequences of $(Y_n)_{n\geq 0}$ are independent. Thus, consider $m <n$ and see that
	\begin{align*}
		\IP &\left( Y_n = 1, Y_m = 1 \right)\\  
		&=   \sum_{A\subset \Lambda^N\cup E^N} \IP ( Y_n = 1 ,Y_m = 1 | \bfhZ_{n-1}^{N} = A) \mathbb{P} (\bfhZ_{n-1}^{N} = A)\\
		&=   \sum_{A\subset \Lambda^N\cup E^N} \IP ( Y_n = 1|\bfhZ_{n-1}^{N} = A) \IP (  Y_m = 1 | \bfhZ_{n-1}^{N} = A ) \mathbb{P} (\bfhZ_{n-1}^{N} = A)\\
		&=   \tfrac{2q_Nf(N)}{N(N-1)}\sum_{A\subset \Lambda^N\cup E^N}  \IP (  Y_m = 1 | \bfhZ_{n-1}^{N} = A )\mathbb{P} ( \bfhZ_{n-1}^{N} = A)\\
		&=   \IP \left( Y_n = 1\right) \IP \left(Y_m = 1 \right),
	\end{align*}
	where we used in the second equality that given $\bfhZ_{n-1}^{N}$ the random variable $Y_n$ only depends on the uniformly chosen edge at time $n$, $\hV^N_n$ and $\hU^N_n$. Since these three objects are independent
	of everything before time $n-1$ the claimed conditional independence follows.

	Finally note that by definition
	\begin{equation*}
		X^{N,\textnormal{op}}_n\geq Y^f_n \quad\text{ if }\quad \hE_{n-1}^{N,d} \geq f(N).\qedhere
	\end{equation*}
\end{proof}

We continue by introducing some further notation.
For fixed $\kappa\in(0,1)$ let $\varepsilon=(1-\kappa)/2$ and recall $n_2:=n_2(N,\kappa,\varepsilon)$ from Lemma~\ref{lem:ConnectivityAndMinimalDegree}.
Let us denote the event that for all times $n\leq n_2$ the minimal degree in the graph $\bfhG_n^N$ is at least $\kappa N$, where $\kappa\in(0,1)$, by
\begin{equation*}
	B^N(\kappa)
	:= \bigcap_{n\leq n_2} \big\{ \hD_n^{N, \textnormal{min}} \geq \kappa N\big\}.
\end{equation*}
By definition, the dynamical deletion graph $(\Lambda^N,\cE^N_n)$ is a subgraph of $\bfhG_n^N$, and thus Lemma~\ref{lem:ConnectivityAndMinimalDegree} implies that $B^N(\kappa)$ holds with high probability for all $\kappa \in (0,1)$.
Furthermore, on $B^N(\kappa)$ the graph $\bfhG_n^N$ is fully connected for $n\leq n_2$ with high probability.

Recall that the number of opinion updates happening between times $\ell$ and $m$ was
\begin{equation*}
	\hS_{[\ell,m]}^{N, \textnormal{op}} = \sum_{n=\ell}^m X_n^{N, \textnormal{op}}
\end{equation*}
and $\hS_{m}^{N, \textnormal{op}} =\hS_{[1,m]}^{N, \textnormal{op}}$. Recalling that $r=(1-q_N)^{-1/2}$, we define
\begin{equation*}
	f_k(N)= \kappa r_N^{-k} N^2-r_N^{-2k} N^2 .
\end{equation*}
and let $(Y_n^{N, k})_{n \geq 0}$ be the random variables constructed in Lemma~\ref{lem:iidcoupling} with respect to $f_k$.
Furthermore, set
\begin{equation*}
	\hT_{[\ell,m]}^{N, k} \defeq \sum_{n=\ell}^m Y_n^{N, k}
\end{equation*}
%\rapm{was: $\Hat{N}_{[l,m]}^{N, k}$}.
and for fixed $\alpha>1$ set $M_N=\lceil r_N^{\alpha} N^2\rceil$.
Note that $\hT_{[r,r+m]}^{N, k} \stackrel{d}{=} \hT_{[1,m]}^{N, k}$. Thus, we write $\hT_{M_N}^{N, k} \defeq \hT_{[1,M_N]}^{N, k}$.

\begin{remark}\label{rem:lower_bound_iid}
	Lemma \ref{lem:iidcoupling} implies that if $\{\hE_m^{N,d} \geq f^k(N) \text{ for all } m \leq n_N\}$ holds with high probability, then with high probability $\hS_{[\ell,n_N]}^{N, \textnormal{op}} \geq \hT_{[\ell,n_N]}^{N, k}$, where $\ell\leq n_N$.
\end{remark}

Recall that $k_0\in \N$ was chosen such that $r_N^{-k}N^2\to 0$ as $N\to \infty$ for all $k>k_0$.
The next lemma provides a useful lower bound for $\widehat T^{N,k}$ and hence for $\widehat S^{N,\textnormal{op}}$.

\begin{lemma}\label{NumberOfJumpsLowerBound}
	Let $c\in (0,1)$. Then it follows for every $k\leq k_0+1$ that
	\begin{equation*}
		\lim_{N \to \infty} \IP \Big(\hT_{M_N}^{N, k} \geq c \kappa r^{\alpha-k}_N N^2 \Big) = 1.
	\end{equation*}
\end{lemma}
\begin{proof}
	First we remark that $\hT_{M_N}^{N, k}=\sum_{n=1}^{M_N}Y^{N, k}_n\sim \text{Bin}\big(M_N,\frac{2q_Nf_k(N)}{N(N-1)}\big)$, i.e.\ a binomial random variable with expected value $\frac{2q_Nf_k(N)M_N}{N(N-1)}$. Note that there exists an $N_0\geq 1$ such that $q_N(\kappa-r_N^{-1})\geq \frac{\kappa}{2}$ for all $N\geq N_0$. This implies as well that for any $k\geq 1$ it holds that $q_N(\kappa-r_N^{-k})\geq \frac{\kappa}{2}$ for all $N\geq N_0$. Recalling $M_N=\lceil r_N^\alpha N^2\rceil$, we see for $N\geq N_0$ that 
	\begin{equation*}
		\frac{q_Nf_k(N)M_N}{N(N-1)}
		= \frac{q_N(\kappa-r_N^{-k}) r_N^{-k} N^2 \lceil r_N^\alpha N^2\rceil}{N(N-1)}\geq \tfrac\kappa2 r_N^{\alpha-k}N^2.
	\end{equation*}
	Thus, letting $c\in(0,1)$, it follows for $N$ large enough that
	\begin{align*}
		\IP\Big(\hT_{M_N}^{N, k}\leq c \frac{2q_Nf_k(N)M_N}{N(N-1)}\Big)
		&\leq \exp\Big(-(1-c)^2 \frac{q_Nf_k(N)M_N}{N(N-1)}\Big)\\
		&\leq \exp\Big(-\frac{(1-c)^2}{2}\kappa r_N^{\alpha-k}N^2\Big)
	\end{align*}
	where we used \cite[Theorem 2.21]{van2016random}.
\end{proof}

The next two propositions will be key to the proof of our final result. By induction, we obtain upper bounds on the hitting times $\tau_k$. The induction base is given in the first proposition.

\begin{proposition}\label{prop:Aone}
	Let $\alpha \in (1,2)$. Then it follows that
	\begin{equation*}
		\lim_{N \to \infty} \IP \bigl( \tau_1 \leq r_N^\alpha N^2 \bigr) = 1
	\end{equation*}
\end{proposition}

\begin{proof}
	Recall $M_N=\lceil r_N^{\alpha} N^2\rceil$ and that $\hS_{M_N}^{N, \textnormal{op}}=\sum_{n=1}^{M_N}X_n^{N, \textnormal{op}}$ denotes the total number of jumps of the simple random walk associated to the opinions $\hZ_n^N$ made until time step $M_N$. First we see 
	that for any $c > 0$
	\begin{equation}\label{eq:Aone1}
		\IP\big(\{\hS_{M_N}^{N, \textnormal{op}} \geq c \kappa r_N^{\alpha-1} N^{2}\}\cap \{\tau_1 > r_N^\alpha N^2\}\big)\to 0 \quad \text{as} \quad N\to\infty,
	\end{equation}
	which follows as a consequence of Lemma~\ref{lem:AsymptoticsForDifferentScale}. This lemma implies that a simple symmetric random walk started in an arbitrary initial state whose total number of jumps is of order $r_N^{\alpha-1}N^{2} \in \omega(N^2)$ must have hit the boundary $0$ or $N$ with high probability as $N\to \infty$. Therefore, it must have passed by $r_N^{-1}N$ or $N-r_N^{-1}N$.
	
	Next note that if $\tau_1 > r_N^{\alpha} N^2$ then it must hold that $\hZ^{N, \textnormal{min}}_n \geq r_N^{-1} N$ for all $n\leq r_N^{\alpha} N^2$.
	Thus, we can conclude that on the event $B^N(\kappa)\cap \{\tau_1 > r_N^{\alpha} N^2\}$ it holds for all $n\leq r_N^{\alpha} N^2$ that
	$$
	\hE_n^{N,d}
	\geq r_N^{-1}N(\kappa N - r_N^{-1}N)
	= f_1(N).
	$$
	Now, Lemma~\ref{lem:iidcoupling} (and the corresponding Remark~\ref{rem:lower_bound_iid}) implies that
	\begin{align*}
		&\IP\big(\{ \hS_{M_N}^{N, \textnormal{op}} \geq c \kappa r_N^{\alpha-1} N^{2}\}\cap B^N(\kappa)\cap \{\tau_1> r_N^{\alpha} N^2\}\big)\\
		&\quad \geq\IP\big(\{\hT_{M_N}^{N, 1}\geq c \kappa r_N^{\alpha-1} N^{2}\}\cap B^N(\kappa)\cap \{\tau_1>r_N^{\alpha} N^2\}\big).
	\end{align*}
	Additionally, Lemma~\ref{NumberOfJumpsLowerBound} with $k=1$ implies that
	\begin{equation*}
		\IP\big( \hT_{M_N}^{N, 1} \geq c \kappa r_N^{\alpha-1} N^{2}\big)\to 1
	\end{equation*}
	and we know that $\IP(B^N(\kappa))\to 1$. These two facts yield that
	\begin{equation*}
		0
		= \lim_{N\to \infty}\IP\big(\{\hS_{M_N}^{N, \textnormal{op}} \geq c \kappa r_N^{\alpha-1} N^{2}\}\cap \{\tau_1> r_N^{\alpha} N^2 \}\big)
		\geq \limsup_{N\to \infty}\IP\big(\tau_1> r_N^{\alpha} N^2\big).\qedhere
	\end{equation*}
\end{proof}

Continuing the induction we prove a uniform bound on the increments of the $\tau_k$.

\begin{proposition}\label{prop:Ak}
	Let $\alpha \in (1,2)$. Define
	\begin{equation*}
		A_\alpha^N (k) \defeq \bigcap_{\ell=1}^k \{ |\tau_\ell - \tau_{\ell-1} | \leq r_N^\alpha N^2\} .
	\end{equation*}
	Then, for any $k \in \N$, it holds that $\lim_{N \to \infty} \IP (A_\alpha^N (k))=1$.
\end{proposition}

\begin{proof}
	We prove the statement by induction.
	The case $k=1$ was proved in Proposition~\ref{prop:Aone}. 
	Now we assume that $\lim_{N \to \infty} \IP (A_\alpha^N (k))=1$ holds for some $k$ and show that then also $\lim_{N \to \infty} \IP (A_\alpha^N (k+1))=1$.
	
	Define the following two stopping times 
	\begin{equation*}
		\tau^{\downarrow}_{k}:=\inf \{n\geq 0: \hZ_n^N=\lfloor r_N^{-k}N \rfloor\} \text{ and }
		\tau^{\uparrow}_{k}:=\inf \{n\geq 0: \hZ_n^N= N-\lfloor r_N^{-k}N \rfloor \}
	\end{equation*}
	and see that $\tau_k=\tau^{\downarrow}_{k}\wedge \tau^{\uparrow}_{k}$. Furthermore we fix $\eta\in(0,1\wedge\tfrac\alpha2)$ and define
	\begin{align*}
		\overline{\sigma}^{\uparrow}_{k}&:=\inf\{n>\tau^{\uparrow}_{k}: \hZ_n^N= N-\lfloor r_N^{-(k-\eta)}N\rfloor\}-\tau^{\uparrow}_{k},\\
		\underline{\sigma}_{k}^{\uparrow}&:=\inf\{n>\tau^{\uparrow}_{k}: \hZ_n^N= N-\lfloor r_N^{-(k+1)}N\rfloor\}-\tau^{\uparrow}_{k},\\
		\overline{\sigma}^{\downarrow}_{k}&:=\inf\{n>\tau^{\downarrow}_{k}: \hZ_n^N= \lfloor r_N^{-(k-\eta)}N\rfloor\}-\tau^{\downarrow}_{k}\text{ and }\\
		\underline{\sigma}_{k}^{\downarrow}&:=\inf\{n>\tau^{\downarrow}_{k}: \hZ_n^N= \lfloor r_N^{-(k+1)}N\rfloor\}-\tau^{\downarrow}_{k},
	\end{align*}
	which are not necessarily stopping times.
	
	To give more intuition to this notation, note that we strive to control trajectories to the boundary $\{0,N\}$, either upwards $(\uparrow)$ or downwards $(\downarrow)$. $\tau_k^\bullet$ marks the hitting time of the random walk corresponding to opinion $1$ of the $k$th level in the respective direction, while the $\sigma_k^\bullet$ denote the waiting times from $\tau_k^\bullet$ to the hitting time of the next \emph{good} level closer to the boundary (underline) or the last \emph{bad} level, closer to the center (overline).
	
	We further set
	\begin{align*}
		\overline{\sigma}_{k}:=\overline{\sigma}^{\uparrow}_{k}\1_{\{\tau_k=\tau_k^{\uparrow}\}}+\overline{\sigma}^{\downarrow}_{k}\1_{\{\tau_k=\tau_k^{\downarrow}\}} \text{ and } \underline{\sigma}_{k}:=\underline{\sigma}^{\uparrow}_{k}\1_{\{\tau_k=\tau_k^{\uparrow}\}}+\underline{\sigma}^{\downarrow}_{k}\1_{\{\tau_k=\tau_k^{\downarrow}\}}.
	\end{align*}
	Technically, all of the defined times depend on $N$ as well. However, we leave out the superscript for ease of notation.
	
	Note that $\underline{\sigma}_{k}\leq \overline{\sigma}_{k}$ and $\underline{\sigma}_{k}\wedge \overline{\sigma}_{k}\leq \lceil r_N^{\alpha}N^{2}\rceil$ together imply that $|\tau_{k+1}-\tau_k|\leq \lceil r_N^{\alpha}N^{2}\rceil$, and thus
	\begin{equation*}
		\IP(A_\alpha^N(k+1))\geq \IP\big(\{\underline{\sigma}_{k}\leq \overline{\sigma}_{k},\underline{\sigma}_{k}\wedge \overline{\sigma}_{k}\leq \lceil r_N^{\alpha}N^{2} \rceil\}\cap A_\alpha^N(k)\big).
	\end{equation*}
	Since, by the induction hypothesis, $A_\alpha^N (k)$ holds w.h.p., it suffices to show that the events $\{\underline{\sigma}_{k}\leq \overline{\sigma}_{k}\}, \{\underline{\sigma}_{k}\wedge \overline{\sigma}_{k}\leq \lceil r_N^{\alpha}N^{2} \rceil \}$ hold w.h.p.\ to complete the induction.\\
	
	First we prove $\lim_{N\to \infty}\IP\big(\underline{\sigma}_{k}\wedge \overline{\sigma}_{k}\leq \lceil r_N^{\alpha}N^{2} \rceil\big)= 1$. It holds that $|\tau_{k+1}-\tau_{k}|\geq \underline{\sigma}_{k}\wedge \overline{\sigma}_{k}$, and thus similar as in the proof of Proposition~\ref{prop:Aone} we see that
	\begin{align*}
		\IP\big(\big\{\underline{\sigma}_{k}\wedge \overline{\sigma}_{k}>  \lceil r_N^{\alpha}N^{2}\rceil \big\}\cap  \big\{ \hS_{[\tau_k,M_N+\tau_k]}^{N,\textnormal{op}} 
		\geq c\kappa r_N^{\alpha-(k+1)}N^{2} \big\} \cap B^N(\kappa)\cap A_\alpha^N(k)\big)\to 0
	\end{align*}
	as $N\to 0$, since $r_N^{\alpha-(k+1)}N^{2}$ is of higher order than $ (r_N^{-(k-\eta)}N)^2$, and thus the delayed random walk  $(\hZ_n^N)_{n\geq 0}$ will hit one of the two boundaries with high probability before time step $\lceil r_N^{\alpha}N^{2} \rceil+\tau_k$ is reached. This follows again by Lemma~\ref{lem:AsymptoticsForDifferentScale}.
	
	Note that on the event $A_\alpha^N(k)$ it holds that $\tau_k\leq k  r_N^{\alpha}N^{2}$, since $|\tau_l-\tau_{l-1}|\leq r_N^{\alpha}N^{2}$ for all $l\leq k$. Therefore, we can conclude that
	\begin{equation*}
		\hE^{N,d}_n\geq f_{k+1}(N) \quad\text{on}\quad B^N(\kappa)\cap A_\alpha^N(k)\cap \{\underline{\sigma}_{k}\wedge \overline{\sigma}_{k}>  r_N^{\alpha}N^{2} \}
	\end{equation*}
	for all $\tau_k\leq n\leq \lceil r_N^{\alpha}N^{2} \rceil+\tau_k$. Again by Lemma~\ref{lem:iidcoupling} and Remark~\ref{rem:lower_bound_iid} it holds that
	\begin{align*}
		\IP & \big(\{\underline{\sigma}_{k}\wedge \overline{\sigma}_{k}> \lceil r_N^{\alpha}N^{2}\rceil\}\cap \big\{\hS_{[\tau_k,M_N+\tau_k]}^{N,\textnormal{op}}
		\geq c\kappa r_N^{\alpha-(k+1)}N^{2} \big\}\cap B^N(\kappa)\cap A_\alpha^N(k) \big)\\
		&\geq \IP\big(\{\underline{\sigma}_{k}\wedge \overline{\sigma}_{k}> \lceil r_N^{\alpha}N^{2} \rceil\}\cap \big\{\hT_{[\tau_k,M_N+\tau_k]}^{N, k+1} %T_Y^{k+1}\big(N, \tau_k\big) 
		\geq c\kappa r_N^{\alpha-(k+1)}N^{2} \big\}\cap B^N(\kappa)\cap A_\alpha^N(k) \big).
	\end{align*}
	Since $(Y^{k+1}_l)_{l\in \N}$ are i.i.d.\ by applying Lemma~\ref{NumberOfJumpsLowerBound} we get that
	\begin{equation*}
		\IP\big(\hT_{[\tau_k,M_N+\tau_k]}^{N,k}%T_Y^{k+1}\big(N, \tau_k\big) 
		\geq c\kappa r_N^{\alpha-(k+1)}N^{2}  \big)
		= \IP\big(\hT_{M_N}^{N, k+1}%T_{Y}^{k+1}(N) 
		\geq c\kappa r_N^{\alpha-(k+1)}N^{2}\big)\to 1
	\end{equation*}
	as $N\to \infty$. This fact together with the fact that $\IP(A_\alpha^N(k)\cap B^N(\kappa))\to 1$  implies that
	\begin{equation}\label{AsymptoticBoundForHittingTimes}
		\lim_{N\to \infty}\IP\big(\underline{\sigma}_{k}\wedge \overline{\sigma}_{k}\leq \lceil r_N^{\alpha}N^{2} \rceil\big)= 1 .
	\end{equation}
	
	Now we prove that
	\begin{equation*}
		\lim_{N\to \infty}\IP\big(\underline{\sigma}_{k}\leq \overline{\sigma}_{k},\underline{\sigma}_{k}\wedge \overline{\sigma}_{k}\leq \lceil r_N^{\alpha}N^{2}  \rceil\big)= 1 .
	\end{equation*}
	One problem is that for finite $N$ the waiting time $\underline{\sigma}_{k}\wedge \overline{\sigma}_{k}$ is not necessarily almost surely finite, i.e.\ $(\hZ_n^N)_{n\geq 0}$ might not hit the defined boundaries because the number of discordant edges may reach zero before that. Thus, we need to construct an auxiliary process $(\widetilde{Z}_n^N)_{n\geq 0}$ which will hit the boundaries. 
	To this end let $\theta:=\inf\{n\geq 0: \widehat E^{N,d}_n=0\}$ and let $(Y'_k)_{k\in\N}$ be independent Bernoulli random variables with $\IP(Y'_k=1)=\frac{1}{2}$, then set
	\begin{equation*}
		\widetilde{Z}_n^N:= \hZ_n^N+\sum_{k=1}^{n-\theta}(2Y'_k-1),
	\end{equation*}
	where we use the convention that if $n-\theta\leq 0$ the sum is $0$. This new process $(\widetilde{Z}_n^N)_{n\geq 0}$ will hit the boundaries almost surely. Furthermore, on the event $\{\underline{\sigma}_{k}\wedge \overline{\sigma}_{k}\leq \lceil r_N^{\alpha}N^{2}\rceil\}$ we know for all $n<\underline{\sigma}_{k}\wedge \overline{\sigma}_{k}+\tau_k$ that $\widehat E^{N,d}_n>0$ and thus, on this event $\widetilde{Z}_n^N=\hZ_n^N$. This yields that
	\begin{equation*}
		\begin{aligned}
			&\IP\big(\widetilde{Z}^N_{\underline{\sigma}_{k}\wedge \overline{\sigma}_{k}+\tau_k}\in \{N-\lfloor r_N^{-(k+1)}N\rfloor,\lfloor r_N^{-(k+1)}N\rfloor \},\underline{\sigma}_{k}\wedge \overline{\sigma}_{k}\leq \lceil r_N^{\alpha}N^{2}  \rceil\big)\\
			&\qquad \qquad =\IP\big(\underline{\sigma}_{k}\leq \overline{\sigma}_{k},\underline{\sigma}_{k}\wedge \overline{\sigma}_{k}\leq \lceil r_N^{\alpha}N^{2}  \rceil\big)   
		\end{aligned}
	\end{equation*}
	Using this equality and subadditivity we can see that it holds
	\begin{align*}
		&\IP\big(\underline{\sigma}_{k}\leq \overline{\sigma}_{k},\underline{\sigma}_{k}\wedge \overline{\sigma}_{k}\leq \lceil r_N^{\alpha}N^{2} \rceil\big)\\
		&\qquad\quad \geq  1-\frac{\lfloor r_N^{-k}N\rfloor-\lfloor r_N^{-(k+1)}N\rfloor}{\lfloor r_N^{-(k-\eta)}N\rfloor-\lfloor r_N^{-(k+1)}N\rfloor} - \IP\big(\underline{\sigma}_{k}\wedge \overline{\sigma}_{k}> \lceil r_N^{\alpha}N^{2} \rceil\big),
	\end{align*}
	Now we note that the convergences in \eqref{AsymptoticBoundForHittingTimes} in particular implies that
	\begin{equation*}
		\IP\big(\underline{\sigma}_{k}\wedge \overline{\sigma}_{k}> \lceil r_N^{\alpha}N^{2}\rceil)\to0
	\end{equation*}
	as $N\to \infty$, and therefore we see that
	\begin{equation*}
		\IP\big(\underline{\sigma}_{k}\leq \overline{\sigma}_{k},\underline{\sigma}_{k}\wedge \overline{\sigma}_{k}\leq \lceil r_N^{\alpha}N^{2}  \rceil\big)\to 1.\qedhere
	\end{equation*}    
\end{proof}

\begin{proof}[Proof of Theorem \ref{thm:consensus}]
	Note that for $1-q_N \in \mathcal{O} (N^{-\delta})$ with $\delta > 0$ it holds that
	\begin{equation*}
		r_N^{-k} N = \mathcal{O} \bigl(N^{-\frac{k \delta}{2}}\cdot N \bigr) = o (1)
	\end{equation*}
	for $k \geq k_0 = \big\lceil \frac{2}{\delta} \big\rceil$.  Therefore, for $N$ large enough, $\tau_{k_0}$ coincides with the hitting time of the boundary,
	\begin{equation*}
		\tau_{k_0} \defeq \min \{n \geq 0 : \widehat{Z}_n^N = 0 \text{ or } \widehat{Z}_n^N = N \}.
	\end{equation*}
	By Proposition~\ref{prop:Ak}, we get that with high probability $\tau_{k_0} \leq k_0 r_N^\alpha N^2$ and by Lemma \ref{lem:ConnectivityAndMinimalDegree} we know that the delayed OV-Model is with high probability connected with minimal degree $\kappa N$ until time 
	$\binom N2\frac{\log(1/(\kappa+\varepsilon))}{2(1-q_N)} > \tau_{k_0}$.
	Therefore, with high probability, the voting dynamics reach consensus while the graph is still connected with a minimal degree of at least $\kappa N$.
\end{proof}

\textbf{Acknowledgement.}
We cordially thank Sascha Franck, Marius A. Schmidt and the audience of the Stochastics Seminar at Haus Bergkranz in Riezlern for their valuable remarks. The authors would like to thank the anonymous referees for their detailed comments and suggestions which helped us to improve this article.\\
MS was supported by the German Research Foundation (DFG)  
Project ID: 531542160.\\
RE was supported by the DFG under Project ID: 443227151.

\bibliographystyle{abbrv}
\bibliography{references} 
\end{document}